\documentclass[preprint]{imsart}

\RequirePackage[OT1]{fontenc}
\RequirePackage{amsthm,amsmath,amsfonts,amssymb}
\RequirePackage[colorlinks]{hyperref}
\RequirePackage{epsfig, graphicx, color}
\RequirePackage{times}
\RequirePackage{latexsym}
\RequirePackage{psfrag}
\RequirePackage{color}

\usepackage{epsf,epsfig,subfigure}
\usepackage{epsf,epsfig}
\usepackage{fancyheadings}
\usepackage{graphics,graphicx}
\usepackage{enumerate}
\usepackage{algorithmicx,algorithm}

\newtheorem{theorem}{Theorem}[section]
\newtheorem{lemma}[theorem]{Lemma}
\newtheorem{condition}[theorem]{Condition}

\newtheorem{corollary}[theorem]{Corollary}

\theoremstyle{definition}

\newtheorem{example}[theorem]{Example}

\theoremstyle{remark}
\newtheorem{remark}[theorem]{Remark}

\newcommand{\G}{\mathcal{G}}
\newcommand{\N}{\mathcal{N}}
\newcommand{\R}{\mathbb{R}}
\newcommand{\PP}{\mathbb{P}}
\newcommand{\E}{\mathbb{E}}

\newcommand{\intersection}{\cap}

\newcommand{\un}{\textrm{union}}

\newcommand{\M}{\mathcal{M}}
\newcommand{\B}{\mathcal{B}}
\newcommand{\Ev}{\mathcal{E}}
\newcommand{\C}{\mathcal{C}}
\newcommand{\F}{\mathcal{F}}

\DeclareMathOperator*{\argmin}{\arg\!\min}
\DeclareMathOperator*{\argmax}{\arg\!\max}
\DeclareMathOperator*{\pa}{Pa}
\DeclareMathOperator*{\dsep}{\textrm{is\ d-separated\ from}}

 \newcommand\independent{\protect\mathpalette{\protect\independenT}{\perp}}
    \def\independenT#1#2{\mathrel{\rlap{$#1#2$}\mkern2mu{#1#2}}}

\algnewcommand\algorithmicinput{\textbf{Input:}}
\algnewcommand\Input{\item[\algorithmicinput]}
\algnewcommand\algorithmicoutput{\textbf{Output:}}
\algnewcommand\Output{\item[\algorithmicoutput]}

\begin{document}

\begin{frontmatter}
\title{High-Dimensional Joint Estimation \\ of Multiple Directed Gaussian Graphical Models}
\runtitle{Joint Estimation of Multiple DAGs}

\begin{aug}
\author{\fnms{Yuhao} \snm{Wang}\thanksref{t1}\ead[label=e1]{yw505@cam.ac.uk}}

\address{Statistical Laboratory, University of Cambridge, Cambridge, UK \\ \printead{e1}}

\author{\fnms{Santiago} \snm{Segarra}\thanksref{t1}\ead[label=e2]{segarra@rice.edu}}

\address{Department of Electrical and Computer Engineering, Rice University, Houston, TX, USA \\ \printead{e2}}

\author{\fnms{Caroline} \snm{Uhler}\ead[label=e3]{cuhler@mit.edu}}

\address{Laboratory for Information \& Decision Systems and Institute for Data, Systems, and Society\\
Massachusetts Institute of Technology, Cambridge, MA, USA \\ \printead{e3}}

\thankstext{t1}{At the time this research was completed, Yuhao Wang and Santiago Segarra were at the Massachusetts Institute of Technology.}
\runauthor{Y. Wang et al.}

\affiliation{University of Cambridge, Rice University and Massachusetts Institute of Technology}

\end{aug}

\begin{abstract}
We consider the problem of jointly estimating multiple related directed acyclic graph (DAG) models based on high-dimensional data from each graph. This problem is motivated by the task of learning gene regulatory networks based on gene expression data from different tissues, developmental stages or disease states. We prove that under certain regularity conditions, the proposed $\ell_0$-penalized maximum likelihood estimator converges in Frobenius norm to the adjacency matrices consistent with the data-generating distributions and has the correct sparsity. In particular, we show that this joint estimation procedure leads to a faster convergence rate than estimating each DAG model separately. As a corollary, we also obtain high-dimensional consistency results for causal inference from a mix of observational and interventional data. 
For practical purposes, we propose \emph{jointGES} consisting of Greedy Equivalence Search (GES) to estimate the union of all DAG models followed by variable selection using lasso to obtain the different DAGs, and we analyze its consistency guarantees. The proposed method is illustrated through an analysis of simulated data as well as epithelial ovarian cancer gene expression data.
\end{abstract}

\begin{keyword}[class=MSC]
\kwd[Primary ]{62F12}
\kwd[; secondary ]{62F30}
\end{keyword}

\begin{keyword}
\kwd{Causal inference}
\kwd{linear structural equation model}
\kwd{high-dimensional statistics}
\kwd{graphical model}
\end{keyword}
\tableofcontents
\end{frontmatter}

\section{Introduction}




Directed acyclic graph (DAG) models, also known as Bayesian networks, are widely used to model causal relationships in complex systems across various fields such as computational biology, epidemiology, sociology, and environmental management~\cite{aguilera2011bayesian, Friedman_2000, Pearl:00, robins2000marginal, spirtes:00}. In these applications we often encounter \emph{high-dimensional} datasets where the number of variables or nodes greatly exceeds the number of observations. While the problem of structure identification for undirected graphical models in the high-dimensional setting is quite well understood~\cite{Ravikumar_2011,meinshausen2006high,friedman2008sparse,cai2011constrained,yuan2006model}, such results are just starting to become available for directed graphical models. The difficulty in identifying DAG models can be attributed to the fact that searching over the space of DAGs is NP-complete in general~\cite{chickering:95}.

Methods for structure identification in directed graphical models can be divided into two categories and hybrids of these categories. \emph{Constraint-based methods}, such as the prominent PC algorithm, first learn an undirected graph from conditional independence relations and in a second step orient some of the edges~\cite{GSSK87, spirtes:00}. \emph{Score-based methods}, on the other hand, posit a scoring criterion for each DAG model, usually a penalized likelihood score, and then search for the network with the highest score given the observations. An example is the celebrated Greedy Equivalence Search (GES) algorithm, which can be used to greedily optimize the $\ell_0$-penalized likelihood such as the Bayesian Information Criterion (BIC)~\cite{chickering2002optimal}. High-dimensional consistency guarantees were recently obtained for the PC algorithm~\cite{KB07} and for score-based methods~\cite{loh2014high,nandy2015high,vandegeer2013}. 

Existing methods have focused on estimating a single directed graphical model. However, in many applications we have access to data from related classes, such as gene expression data from different tissues, cell types or states~\cite{MacBasSat15, shalek2014single}, different developmental stages~\cite{arbeitman2002gene}, different disease states~\cite{tothill2008novel}, or from different perturbations such as knock-out experiments~\cite{dixit2016perturb}. In all these applications, one would expect that the underlying regulatory networks are similar to each other, since they stem from the same species, individual or cell type, but also have important differences that drive differentiation, development or a certain disease. This raises an important statistical question, namely how to jointly estimate related directed graphical models in order to effectively make use of the available data.

Various methods have been proposed for jointly estimating \emph{undirected} Gaussian graphical models. To preserve the common structure, Guo et al.~\cite{guo2011joint} suggested to use a hierarchical penalty and Danaher et al.~\cite{danaher2014joint} suggested the use of a generalized fused lasso or group lasso penalty. While both approaches achieve the same convergence rate as the individual estimators, Cai et al.~\cite{cai2015joint} were able to improve the asymptotic convergence rate of joint estimation using a weighted constrained $\ell_\infty / \ell_1$ minimization approach. Bayesian methods have been proposed for this problem as well~\cite{peterson2015bayesian}. Related works also include~\cite{mohan2014node}, where it is assumed that the networks differ only locally in a few nodes and~\cite{kolar2010estimating, song2009keller}, where the assumption is that the networks are ordered and related by continuously changing edge weights.

In this paper, we propose a framework based on $\ell_0$-penalized maximum likelihood estimation for jointly estimating related \emph{directed} Gaussian graphical models. We show that the joint $\ell_0$-penalized maximum likelihood estimator (MLE) achieves a faster asymptotic convergence rate as compared to the individual estimators. In addition, by viewing interventional data as data coming from a related network, we show that the interventional BIC scoring function proposed in~\cite{hauser2012characterization} can be obtained as a special case of the joint $\ell_0$-penalized maximum likelihood approach presented here. Our theoretical consistency guarantees also explain the empirical findings of~\cite{hauser2012characterization}, namely that estimating a DAG model from interventional data usually leads to better recovery rates as compared to estimating  a DAG model from the same amount of purely observational data. These theoretical results are based on the global optimum of $\ell_0$-penalized maximum likelihood estimation. To overcome the computational bottleneck of this optimization problem we propose a greedy approach (\emph{jointGES}) for solving this problem by extending GES~\cite{chickering2002optimal} to the joint estimation setting. We analyze its properties from a theoretical point of view and test its performance on synthetic data and gene expression data from epithelial ovarian cancer.

The remainder of this paper is structured as follows. In Section~\ref{sec:bg}, we review some relevant background related to DAG models and introduce notation for the joint DAG estimation problem studied in this paper. In Section~\ref{sec:framework}, we present the joint $\ell_0$-penalized maximum likelihood estimator and jointGES, an adaptation of GES for solving this optimization problem. Section~\ref{sec:theory} establishes results regarding the statistical consistency of the $\ell_0$-penalized MLE and jointGES. Section~\ref{sec:interventions} presents the implications for learning DAG models from a mix of observational and interventional data. In Section~\ref{sec:exp}, we illustrate the performance of our proposal in a simulation study and an application to the analysis of gene expression data. We conclude with a short discussion in Section~\ref{sec:disscusion}. The proofs of supporting results are contained in the Appendix.

\section{Preliminaries} \label{sec:bg}

In Section~\ref{Ss:dags_and_sems} we introduce DAG models, in particular linear structural equation models, and discuss statistical features enjoyed by random vectors following these models. In Section~\ref{sec:obsl0} we briefly review existing approaches for estimating a \emph{single} directed graphical model from observational data. Finally, Section~\ref{Ss:collection_dags} describes a setting where \emph{multiple} related directed graphical models exist.

\subsection{Directed acyclic graphs and linear structural equation models}\label{Ss:dags_and_sems}

Let $\G = (V, E)$ denote a DAG with vertices $V = [p] = \lbrace 1, \cdots, p \rbrace$ and directed edges $E \subseteq V \times V$, where $|\G|$ denotes the cardinality of $E$. Let $A \in \R^{p \times p}$ be the adjacency matrix specifying the edge weights of the underlying DAG $\G$, i.e., $A_{ij} \neq 0$ if and only if $(i,j) \in E$. Also, let $\epsilon \sim \N(0, \Omega)$ denote a $p$-dimensional multivariate Gaussian random variable with zero mean and diagonal covariance matrix $\Omega$. In this work, we assume that the observed random vector $X = (X_1, \cdots, X_p) \in \R^p$ is generated according to the following linear structural equation model (SEM).
\begin{align}\label{eq:SEM_model}
X = A^T X + \epsilon.
\end{align}
Hence $X$ follows a multivariate Gaussian distribution with zero mean and covariance matrix ${\Sigma}$, where the inverse covariance (or \emph{precision}) matrix $\Theta = \Sigma^{-1}$ is given by
\begin{align} \label{eq:precmat}
\Theta = (I - A) \Omega^{-1} (I - A)^T.
\end{align}
Let $\pa_j(\G)$ denote the parents of node $j$ in $\G$;then it follows from (\ref{eq:SEM_model}) that the distribution of $X$ factorizes as
\begin{align*}
\PP(X) = \prod_{j=1}^p \PP(X_j \vert X_{\pa_j(\G)}).
\end{align*}
Such a factorization of $\mathbb{P}$ according to $\mathcal{G}$ is equivalent to 
the \emph{Markov assumption} with respect to $\mathcal{G}$~\cite[Theorem 3.27]{LAU96}. 
Formally, given $j,k \in V$ and an arbitrary subset of nodes $S \subset V \setminus \{j,k\}$, then
\begin{align}\label{eq:def_markov}
j \dsep k \;\vert\;S \;\textrm{ in } \mathcal{G} \quad \Rightarrow \quad X_j \independent X_k \vert X_S \;\textrm{ in } \mathbb{P}.
\end{align}
If the implication \eqref{eq:def_markov} holds bidirectionally, then $\PP$ is said to be \emph{faithful} \cite{spirtes:00} with respect to $\G$. 
%
Note that there exist DAGs $\mathcal{G}_1$ and $\mathcal{G}_2$ that encode the same d-separations and hence the same conditional independence relations. Such DAGs are said to belong to the same \emph{Markov equivalence class}. 
%

A consequence of the acyclicity of $\G$ is that there exists at least one permutation $\pi$ of $[p]$ such that $A_{ij}=0$ for all $\pi(i) \geq \pi(j)$. 
Putting it differently, if the rows and columns of $A$ are reordered according to $\pi$, then the resulting matrix is strictly upper triangular. 
Hence, if such a permutation $\pi$ is known a priori, one can obtain the SEM parameters $(A, \Omega)$ from $\Theta$ according to the following steps [cf.~\eqref{eq:precmat}]. First, we reorder $\Theta$ according to $\pi$. 
Then, we perform on the reordered $\Theta$ an upper-triangular-plus-diagonal Cholesky decomposition to obtain $(A', \Omega')$. 
Finally, we revert the ordering by permuting the rows and columns of $A'$ and $\Omega'$ according to $\pi^{-1}$ and obtain the sought $(A, \Omega)$. 
For an arbitrary permutation $\pi$ and a given $\Theta$, we denote by $(A_\pi, \Omega_\pi)$ the Cholesky decomposition parameters obtained from the procedure just described. 
Alternatively, one can obtain $(A_\pi, \Omega_\pi)$ by solving $p$ linear regressions [cf.~\eqref{eq:SEM_model}]. 
More precisely, we can obtain each column of $A_\pi$ by regressing $X_j$ only on those $X_i$ such that $\pi(i) < \pi(j)$ for all $j$. 
Once $A_\pi$ is obtained, one can estimate the variance of $\epsilon$ in \eqref{eq:SEM_model} to get $\Omega_\pi$.
In the remainder of the paper, we denote by $(A_0, \Omega_0)$ and $\Theta_0$ the true parameters of the data-generating SEM and the associated precision matrix, respectively. Moreover, we denote by $(A_{0\pi}, \Omega_{0\pi})$ the SEM parameters obtained from the described procedure when the true precision matrix $\Theta_0$ is used. Notice that $(A_0, \Omega_0) = (A_{0\pi}, \Omega_{0\pi})$ if $\pi$ is \emph{any} permutation consistent with the true underlying DAG $\G_0$.
The DAG $\G_\pi$ corresponding to the non-zero entries of $A_\pi$ is known as the \emph{minimal} I-MAP (independence map) with respect to $\pi$. 
The minimal I-MAP with the fewest number of edges is called minimal-edge I-MAP~\cite{vandegeer2013}. 
If $\PP$ is faithful with respect to a DAG $\G$, then $\G$ is a minimal-edge I-MAP of $\PP$~\cite{RU13, vandegeer2013}.

Furthermore, it has been shown in~\cite{judea1991equivalence} that all DAGs in a Markov equivalence class share the same \emph{skeleton} -- i.e., the set of edges when directions are ignored -- and \emph{v-structures}. A {v-structure} is a triplet $(j, k, \ell) \subseteq V$ such that $(j, k), (\ell,k) \in E$ but $j$ and $\ell$ are not connected in either direction. This motivates the representation of a Markov equivalence class as a \emph{completely partially directed acyclic graph} (\emph{CPDAG}), which is a graph containing both directed and undirected edges~\cite{CPDAG}. A directed edge means that all DAGs in the Markov equivalence class share the same direction for this edge whereas an undirected edge means that both directions for that specific edge are present within the class.
In the same way, one can represent a subset of a Markov equivalence class via a \emph{partially directed acyclic graph} (\emph{PDAG}), where the directions of the edges are only determined by the graphs within the subset. 
In particular, some undirected edges in a CPDAG would become directed edges in a PDAG representing a subset of the class. 
Notice that both DAGs and CPDAGs are special cases of PDAGs, where the former represents a single graph and the latter represents the whole equivalence class.

To consistently estimate causal DAG models in high dimensions, the $\ell_0$-penalized maximum likelihood estimation approach~\cite{vandegeer2013}, the high-dimensional PC method~\cite{KB07} and the ARGES method~\cite{nandy2015high} have been proposed. These methods have high-dimensional guarantees under different conditions, and are thus not directly comparable. In particular, the theoretical guarantees of $\ell_0$-penalized maximum likelihood estimation requires the so-called ``beta-min'' condition~\cite{vandegeer2013}; the high-dimensional PC algorithm requires the ``strong faithfulness'' condition~\cite{KB07}; and ARGES requires the ``strong faithfulness'' condition as well as additional conditions. 
For further discussions on the strength of different conditions, especially the ``beta-min'' and ``strong faithfulness'' conditions, we refer the readers to Remark~\ref{rk:soa} and \cite[Section~4.3.2]{vandegeer2013}.

\subsection{$\ell_0$-penalized maximum likelihood estimation for a single DAG model} \label{sec:obsl0}

We denote by $\hat{X} \in \R^{n \times p}$ the observed data, where each row of $\hat{X}$ represents a realization of the random vector $X$. We say that we are in the \emph{low-dimensional} setting if asymptotically $p$ remains a constant as $n \to \infty$. By contrast, whenever $p \to \infty$ as $n \to \infty$, we say that we are in the \emph{high-dimensional} setting. Assuming faithfulness, Chickering~\cite{chickering2002optimal} shows that GES outputs a consistent estimator in the low-dimensional setting by optimizing the following objective -- also known as the Bayesian information criterion (BIC) --
\begin{align}\label{eq:l0_max_likelihhod}
(\hat{A}, \hat{\Omega}) := \argmax_{ A \in \mathcal{A}, \, \Omega \in \mathcal{D}_+  } \; \ell_n (\hat{X}; A, \Omega) - \lambda^2 \Vert A \Vert_0,
\end{align}
where $\lambda^2 = \frac{1}{2} \frac{\log n}{n}$, $\mathcal{A}$ denotes the set of all valid adjacency matrices associated with DAGs, $\mathcal{D}_+$ is the set of non-negative diagonal matrices, and $\ell_n$ is the likelihood function
\begin{align} \label{eq:likelihood}
\ell_n (\hat{X}; A, \Omega) \; & := \; - \textrm{trace} \left(\frac{\hat{X}^T \hat{X}}{n} \cdot (I - A) \Omega^{-1} (I - A)^T \right) \nonumber\\
 &\qquad + \log\det \left((I - A) \Omega^{-1} (I - A)^T\right).
\end{align}
In the high-dimensional setting, van de Geer and B{\"u}hlmann \cite{vandegeer2013} give consistency guarantees for the global optimum of~\eqref{eq:l0_max_likelihhod} when the collection $\mathcal{A}$ is further constrained to contain only adjacency matrices with at most $d$ incoming edges for each node, where $d  = \mathcal{O}(n / \log p)$. More precisely, they show that there exists some parameter $\lambda^2 \asymp \frac{\log p}{n}$ such that the optimum $(\hat{A}, \hat{\Omega})$ in \eqref{eq:l0_max_likelihhod} converges in Frobenius norm to $(A_{0\hat{\pi}}, \Omega_{0\hat{\pi}})$ for increasing $n$ and $p$, where $\hat{\pi}$ is a permutation consistent with $\hat{A}$, i.e.,  
\begin{align}\label{eq:error_frobenius_single}
\Vert \hat{A} - A_{0\hat{\pi}} \Vert_F^2 + \Vert \hat{\Omega} - \Omega_{0\hat{\pi}} \Vert_F^2 = \mathcal{O}\left(\lambda^2 \vert \G_0 \vert \right).
\end{align}
Notice, however, that \eqref{eq:error_frobenius_single} does not guarantee statistical consistency since $\hat{\pi}$ need not be a permutation consistent with the true underlying DAG. 
Moreover, \eqref{eq:error_frobenius_single} does not hold for \emph{every} permutation $\hat{\pi}$ consistent with $\hat{A}$, but \cite{vandegeer2013} shows the existence of at least one such permutation.
In addition, it is shown in  \cite{vandegeer2013} that the number of non-zero elements in $\hat{A}$, $A_{0\hat{\pi}}$, and $A_0$ are all of the same order of magnitude, i.e., $|\hat{\G}| \asymp |\G_{0\hat{\pi}}| \asymp |\G_0|$.

\subsection{Collection of DAGs}\label{Ss:collection_dags}

Consider the setting where not all the observed data comes from the same DAG, but rather from a collection of DAGs $\{\G^{(k)} = (V, E^{(k)})\}_{k=1}^K$ that share the same node set $V = [p]$. In addition, we assume that all DAGs in a collection are consistent with some permutation $\pi$. This precludes a scenario where $(i,j) \in E^{(k)}$ and $(j,i) \in E^{(k')}$ for some $k \neq k'$. This is a reasonable assumption in, e.g., the analysis of gene expression data, where regulatory links may appear or disappear, but they in general do not change direction.

Denote by $\{(A^{(k)}, \Omega^{(k)})\}_{k=1}^K$ a set of SEMs on the $K$ DAGs $\{\G^{(k)}\}_{k=1}^K$ and by $\{ \hat{X}^{(k)} \}_{k=1}^K$ the data generated from each SEM, where we observe $n_k$ realizations for each DAG $\G^{(k)}$. In this way, each row of the data matrix $\hat{X}^{(k)}  \in \R^{n_k \times p}$ corresponds to a realization of the random vector $X^{(k)}$ defined as
\begin{align*}
X^{(k)} = {A^{(k)}}^T X^{(k)}   + \epsilon^{(k)} \quad\textrm{with}\quad \epsilon^{(k)} \sim \N(0, \Omega^{(k)}).
\end{align*}

Collections of DAGs arise for example naturally when considering data from \emph{perfect} (also known as \emph{hard}) \emph{interventions}~\cite{EGS05}. Consider a non-intervened DAG $\G$ with SEM parameters $(A, \Omega)$ [cf.~\eqref{eq:SEM_model}]. Then a perfect intervention on a subset of nodes $I_k \subset V$ gives rise to the interventional distribution
%
\begin{align*}
X^{I_k}= {A^{I_k}}^T X^{I_k}   + \epsilon^{I_k} \quad\textrm{with}\quad \epsilon^{I_k} \sim \N(0, \Omega^{I_k}),
\end{align*}
where $A^{I_k}_{ij} = 0$ if $j \in {I_k}$ and $A^{I_k}_{ij} = A_{ij}$ otherwise, and the diagonal matrix $\Omega^{I_k}$ satisfies $\Omega^{I_k}_{ii} = \Omega_{ii}$ if $i \not\in I_k$~\cite{hauser2012characterization,hauser2015jointly}. We denote the DAG given by the non-zero entries of $A^{I_k}$ by $\G^{I_{k}}$.

In accordance with the notation introduced in Section~\ref{Ss:dags_and_sems}, we denote by $\G_0^{(k)}$ and $(A_0^{(k)}, \Omega_0^{(k)})$ the true data-generating DAG and SEM parameters for class $k$, respectively, and by $\pi_0$ a permutation that is consistent with $A_0^{(k)}$ for all classes $k\in[K]$. 
Moreover, we denote by $\G_{0\pi}^{(k)}$ and $(A_{0\pi}^{(k)}, \Omega_{0\pi}^{(k)})$ the DAG and SEM parameters obtained from the Cholesky decomposition of the true precision matrix $\Theta_0^{(k)}$ when permuted by $\pi$. 
We denote by $\Sigma_0^{(k)}$ the true covariance matrix of the SEM for class $k$, i.e., the inverse of ${\Theta_0^{(k)}}$.
Finally, we define $\G_0^{\un}$ as the union of all $\G_0^{(k)}$ -- i.e., an edge appears in $\G_0^{\un}$ if it appears in \emph{any} $\G_0^{(k)}$ -- and $\G_{0\pi}^{\un}$ as the union of $\G_{0\pi}^{(k)}$.
For interventional data, we use $(A_0^{I_k}, \Omega_0^{I_k})$ to denote the true SEM parameters after intervening on targets~$I_k$.

\section{Joint estimation of multiple DAGs} \label{sec:framework}

We first present a penalized maximum likelihood estimator that is the natural extension of \eqref{eq:l0_max_likelihhod} for the case where a collection of DAGs is being estimated. Since this involves minimizing $\left\|\cdot\right\|_0$, we then discuss a greedy approach that alleviates the computational complexity of this estimator.

\subsection{Joint $\ell_0$-penalized maximum likelihood estimator} 
\label{sec:obj}

With $d$ denoting a pre-specified sparsity level and $w_k = n_k/n$ indicating the proportion of observed data from DAG $k$, we propose the following estimator:
\begin{align} \label{obj:jointl0}
\Big\{\hat{\pi}, & \{(\hat{A}^{(k)}, \hat{\Omega}^{(k)})\}_{k=1}^K \Big\} \nonumber\\
 & := \argmax_{\pi, \{(A^{(k)}, \Omega^{(k)})\}_{k=1}^K} \;\; \sum_{k=1}^K w_k \ell_{n_k} (\hat{X}^{(k)}; A^{(k)}, \Omega^{(k)}) - \lambda^2 \bigg\| \sum_{k=1}^K |A^{(k)}| \bigg\|_0 \\
& \qquad\quad\;\textrm{subject to} \qquad A^{(k)} \in \mathcal{A}_{\pi}, \,\,\,  \| A^{(k)} \|_{\infty, 0} \leq d, \,\,\, {\Omega}^{(k)} \in \mathcal{D}_+ \,\,\, \forall k, \nonumber
\end{align}
where $\mathcal{A}_{\pi}$ is the set of all adjacency matrices consistent with permutation $\pi$ and the matrix norm $\| \cdot \|_{\infty, 0}$ computes the maximum $\ell_0$-norm across the rows of the argument matrix. The optimization problem in \eqref{obj:jointl0} seeks to maximize a weighted log-likelihood of the observations (where more weight is given to SEMs with more realizations) penalized by the support of the union of all estimated DAGs. To see why this is true, notice that $\| \sum_{k=1}^K |A^{(k)}|\|_0$ counts the number of $(i,j)$ entries for which $A^{(k)}_{ij} \neq 0$ for at least one graph $k$. This penalization on the union of estimated DAGs promotes overlap in the supports of the different $A^{(k)}$.
Regarding the constraints in \eqref{obj:jointl0}, the first constraint imposes that all estimated DAGs are consistent with the same permutation $\pi$, which is itself an optimization variable. This constraint is in accordance with our assumption in Section~\ref{Ss:collection_dags} and drastically reduces the search space of DAGs. The second constraint ensures that the maximum in-degree in all graphs is at most $d$, and the last constraint imposes the natural requirement that all noise covariances are diagonal and non-negative.

Notice that \eqref{obj:jointl0} is a natural extension of \eqref{eq:l0_max_likelihhod}. Indeed, for the case $K=1$ the objective in \eqref{obj:jointl0} immediately boils down to that in \eqref{eq:l0_max_likelihhod}. Moreover, when there is only one graph and $\pi$ can be selected freely, the constraint $A^{(1)} \in \mathcal{A}_{\pi}$ is effectively identical to $A^{(1)} \in \mathcal{A}$, i.e., the constraint in \eqref{eq:l0_max_likelihhod}. Finally, observe that in \eqref{obj:jointl0} we have included the additional maximum in-degree constraint required in the high-dimensional setting [cf.~discussion after~\eqref{eq:likelihood}].

\subsection{JointGES: Joint greedy equivalence search} \label{sec:jges}

\begin{algorithm}[!t]
	\caption{\;JointGES for joint $\ell_0$-penalized maximum likelihood estimation of multiple DAGs.}\label{alg:jointl0}
	\begin{algorithmic}[1]
		\Input{\ \ \ Collection of observations $\hat{X}^{(1)}\in\mathbb{R}^{{n_1} \times p}, \cdots, \hat{X}^{(K)}\in\mathbb{R}^{{n_K} \times p}$, sparsity bound $d$, penalization parameters $\lambda_1$ and $\lambda_2$}
		\Output{Collection of weighted adjacency matrices $\hat{A}^{(1)}, \cdots, \hat{A}^{(K)}$}
		\State Apply GES to find $\hat{\G}^\un$, an approximate solution to the following optimization problem
		\begin{align}\label{E:alg_step_1}
		\begin{split}
		\argmin_{\G} \;&\; \sum_{j=1}^p \left( \sum_{k=1}^K w_k \left[\min_{a \in \R^{\vert{\pa}_j(\G) \vert}} \log \left( \Vert \hat{X}_j^{(k)} - \hat{X}^{(k)}_{{\pa}_{j}(\G)} \, a \Vert_2^2 \right) \right] + \lambda_1^2 \vert {\pa}_j(\G) \vert \right) \\
		\textrm{subject to} & \quad \max_j \vert {\pa}_j(\G) \vert \leq d 
		\end{split}
		\end{align}
		\State Estimate the weighted adjacency matrices $\{\hat{A}^{(k)}\}_{k=1}^K$ consistent with $\hat{\G}^\un$ by solving $Kp$ sparse regressions of the form
		\begin{align*}
		\begin{split}
		\hat{a}_j^{(k)} & = \argmin_{a \, | \, \textrm{supp}(a) \subseteq \pa_j(\hat{\G}^\un)} \frac{1}{n_k} \Vert \hat{X}_j^{(k)} - \hat{X}^{(k)} a \Vert_2^2 + \lambda^2_2 \Vert a \Vert_1.
		\end{split}
		\end{align*}
	\end{algorithmic}
\end{algorithm}

The $\ell_0$ norm as well as the optimization over all permutations $\pi$ render the problem of~\eqref{obj:jointl0} non-convex, thus, hard to solve efficiently. 
In this section, we present a greedy approach to find a computationally tractable approximation to a solution to \eqref{obj:jointl0}. The algorithm, which we term \emph{JointGES}, is succinctly presented in Algorithm~\ref{alg:jointl0} and consists of two steps.

In the first step of Algorithm~\ref{alg:jointl0} we recover $\hat{\G}^\un$, our estimate of the union of all the DAGs to be inferred.
We do this by finding an approximate solution to \eqref{E:alg_step_1} via the implementation of GES~\cite{chickering2002optimal}. The objective (scoring function) in \eqref{E:alg_step_1} consists of two terms. The first term is given by the sum of the log-likelihoods of the achievable residues when regressing the $j$th column of $X^{(k)}$, denominated as $X_j^{(k)}$, on $X_{\pa_j(\G)}^{(k)}$ for each node $j$ and DAG $k$. 
In \cite{vandegeer2013}, van de Geer and B{\"u}hlmann show that if we keep the underlying DAG $\G$ fixed, the maximum likelihood estimator proposed in~\eqref{eq:l0_max_likelihhod} is equivalent to optimizing $\sum_{j=1}^p \left(\min_{a \in \R^{\vert{\pa}_j(\G) \vert}} \log \left( \Vert \hat{X}_j - \hat{X}_{{\pa}_{j}(\G)} \, a \Vert_2^2 \right) \right)$.
Thus, the first term in~\eqref{E:alg_step_1} corresponds to the first term in the objective of~\eqref{obj:jointl0}.
The second term penalizes the size of the parent set of each node in the graph to be recovered, effectively penalizing the number of edges in the graph.
In this way, the scoring function in \eqref{E:alg_step_1} promotes a sparse $\G$ in the same way that the objective of \eqref{obj:jointl0} promotes the union of all $K$ recovered graphs to have a sparse support. 
Additionally, it is immediate to see that the scoring function in \eqref{E:alg_step_1} is \emph{decomposable}~\cite{chickering2002optimal}, a key feature that enables the implementation of GES to find an approximate solution. 
Once we have obtained the union of all sought DAGs $\hat{\G}^\un$ from step 1, in the second step of our algorithm we estimate the DAGs $\hat{\G}^{(1)}, \cdots, \hat{\G}^{(K)}$ by searching over the subDAGs of $\hat{\G}^\un$. 
More precisely, for each node $j$ we estimate its parents in $\hat{\G}^{(k)}$ by regressing $X_j^{(k)}$ on $X_{{\pa}_j(\hat{\G}^\un)}^{(k)}$ using lasso, where the support of $\hat{a}_j^{(k)}$ corresponds to the set of parents of $j$ in $\hat{\G}^{(k)}$. 

To summarize, Algorithm~\ref{alg:jointl0} recovers $K$ DAGs by first estimating the union of all these DAGs $\hat{\G}^\un$ using GES and then inferring the specific weight adjacency matrices $\hat{A}^{(k)}$ via a lasso regression, while ensuring consistency with the previously estimated~$\hat{\G}^\un$.


\section{Consistency guarantees}\label{sec:theory}

The main goal of this section is to provide theoretical guarantees on the consistency of the solution to Problem~\eqref{obj:jointl0} in the high-dimensional setting.
Our main result is presented in Theorem~\ref{thm:l0}; in Section~\ref{Ss:consistency_milder_conditions} we present a laxer statement of consistency based on milder conditions.

\subsection{Statistical consistency of the joint $\ell_0$-penalized MLE} \label{sec:maintheory}

A series of conditions must hold for our main result to be valid. We begin by stating these conditions followed by the formal consistency result in Theorem~\ref{thm:l0}. The rationale behind these conditions and their implications are discussed after the theorem in Section~\ref{Sss:conditions_for_theorem}.

\begin{condition} \label{cd1}
	All DAGs $\G_0^{(1)}, \cdots, \G_0^{(K)}$ are minimal-edge I-MAPs.
\end{condition}

\begin{condition} \label{cd3}
	There exists a constant $\sigma_0^2$ that bounds the variance of all the observed processes, i.e., $\max_{k,i} [\Sigma_0^{(k)}]_{ii} \leq \sigma_0^2$.
\end{condition}

\begin{condition} \label{cd4}
	The smallest eigenvalues of all $\Sigma_0^{(k)}$ are non-zero, i.e. $\min_k \Lambda_{\min}(\Sigma_0^{(k)}) = {\Lambda_{\min}} > 0$.
\end{condition}

\begin{condition}\label{cd5}
	There exists some constant $\alpha$ such that, for all $k$, the maximum allowable in-degree $d$ in the objective function~\eqref{obj:jointl0} is bounded as $d \leq \alpha n_k / \log p$.
\end{condition}

\begin{condition}\label{cd6}
	For all $\pi$ and $j$ there exist some constants $\tilde{\alpha}$ and $c_s > 2$ such that
	\begin{align*}
	|{\pa}_{j}(\G_{0\pi}^{\un})| + c_s \leq \tilde{\alpha} \left( \min \left\lbrace \left( \frac{n}{K^7 (\log p)^3} \right)^{\frac{1}{3}}, \frac{n}{K^7 (\log n)^2 \log p} \right\rbrace \right).
	\end{align*}
\end{condition}

\begin{condition} \label{cd9}
	The number of DAGs $K$ satisfies $K = o(\log p)$ and the amount of data associated with each DAG is comparable in the sense that $n_1 \asymp n_2 \asymp \cdots \asymp n_K$.
\end{condition}

\begin{condition}\label{cd7}
	There exists some constant $c_t>0$ such that $\vert \G_{0\pi}^{\un} \vert \leq c_t \sum_{k=1}^K w_k \vert \G_{0 \pi}^{(k)} \vert$ for any permutation $\pi$.
\end{condition}

\begin{condition} \label{cd8}
	There exist constants $\eta_0$ and $\eta_1$ such that $0 \leq \eta_1 < 1$, $0 < \eta_0^2 < (1 - \eta_1) / c_t$, and 
	\begin{align}\label{E:noise_level_condition}
	\sum_{i,j} \mathbf{1} \left\{ \left| [A_{0\pi}^{(k)}]_{i,j} \right| > \frac{\sqrt{\log p / n}}{\eta_0} \left(\sqrt{p / \vert \G_0^\un\vert} \vee 1\right) \right\} \geq (1-\eta_1) | \G_{0 \pi}^{(k)}|,
	\end{align}
	for all permutations $\pi$ and graphs $k\in [K]$, where $\mathbf{1}\{ \cdot \}$ denotes the indicator function and $c_t$ is as in Condition \ref{cd7}. 
\end{condition}

With the above conditions in place, the following result can be shown.

\begin{theorem}\label{thm:l0}
	If Conditions~\ref{cd1}-\ref{cd8} hold and $\lambda$ is chosen such that 
	$$\lambda^2 \asymp \frac{\log p}{n} \left( \frac{p}{\vert \G_0^\un \vert} \vee 1 \right),$$
	then there exists a constant \,$c>0$\,, that depends on $c_s$, such that with probability \,$1 - \exp(-cp)$\, the solution to~\eqref{obj:jointl0} satisfies
	\begin{align}\label{E:convergence_main_theorem}
	\sum_{k=1}^K w_k \Vert \hat{A}^{(k)} - A_{0 \hat{\pi}}^{(k)} \Vert_F^2 + 
	\sum_{k=1}^K w_k \Vert \hat{\Omega}^{(k)} - \Omega_{0 \hat{\pi}}^{(k)} \Vert_F^2 = \mathcal{O}\left(\lambda^2 \vert \G_0^\un\vert\right).
	\end{align}
	Furthermore, denoting by $\hat{\G}$ the union of the graphs $\hat{\G}^{(k)}$ associated with the $K$ recovered adjacency matrices $\hat{A}^{(k)}$, we have that
	\begin{align}\label{E:edges_main_theorem}
	\vert \hat{\G} \vert \asymp \vert \G_{0 \hat{\pi}}^\un \vert \asymp \vert \G_0^\un \vert.
	\end{align}
\end{theorem}

The proof of Theorem~\ref{thm:l0} is given in Appendix~\ref{sec:l0proof}. 
To intuitively grasp the result in the above theorem, assume that the number of edges in $\G_0^\textrm{union}$ is proportional to the number of nodes $p$ so that $\lambda^2 \vert \G_0^\un \vert \to 0$ for increasing $n$ as long as $n > p\log p$. 
Hence, under these conditions, \eqref{E:convergence_main_theorem} guarantees that for the recovered permutation $\hat{\pi}$, the estimated adjacency matrix $\hat{A}^{(k)}$ converges to $A_{0 \hat{\pi}}^{(k)}$ in Frobenius norm for all $k$. 
This not only implies that both adjacency matrices have similar structure, but also that the edge weights are similar. Moreover, from \eqref{E:edges_main_theorem} it follows that the number of edges in the estimated graph $\hat{\G}$, i.e., $| \hat{\G} |$ is similar to the number of edges in the union of all minimal I-MAPs with permutation $\hat{\pi}$, i.e., $| \G_{0 \hat{\pi}}^\un |$. More importantly, $| \hat{\G} |$ is also similar to the number of edges in the true union graph $\vert \G_0^\un \vert$.
Despite these guarantees, it should be noted that similar to the results in~\cite{vandegeer2013}, the permutation $\hat{\pi}$ need not coincide with the permutation $\pi$ of the true graphs to be recovered.

We now assess the benefits of performing joint estimation of the $K$ DAGs as opposed to estimating them separately. 
To do so, we compare the guarantees in Theorem~\ref{thm:l0} to those developed in~\cite{vandegeer2013} for separate estimation.
The application of the consistency bound reviewed in \eqref{eq:error_frobenius_single} yields that for the separate estimation of $K$ DAGs, when we are in the setting where all $K$ DAGs are highly overlapping (cf.~Condition~\ref{cd7}), by choosing $\lambda$ such that 
$\lambda^2 \asymp \frac{\log p}{n} \left( \frac{p}{\vert \G_0^\un \vert} \vee 1 \right)$, one can guarantee that
\begin{align}\label{E:guarantee_separate_estimation}
\sum_{k=1}^K w_k  \Vert \hat{A}^{(k)} - A_{0 \hat{\pi}^{(k)}}^{(k)} \Vert_F + \sum_{k=1}^K w_k \Vert \hat{\Omega}^{(k)} - \Omega_{0 \hat{\pi}^{(k)}}^{(k)} \Vert_F^2  = \mathcal{O}\left(K \lambda^2 \max_{k \in [K]} \vert \G_0^{(k)} \vert \right),
\end{align}
where it should be noted that in the separate estimation the recovered permutation $\hat{\pi}^{(k)}$ can vary with $k$.
A direct comparison of \eqref{E:convergence_main_theorem} and \eqref{E:guarantee_separate_estimation} reveals that performing joint estimation improves the accuracy by a factor of $K$ from $\Omega(K \frac{\log p}{n})$ to $\Omega(\frac{\log p}{n})$. Hence, for joint estimation the accuracy scales with the total number of samples $n$, showing that our procedure yields maximal gain from each observation, even if the data is generated from $K$ different DAGs. 
Moreover, the result in \eqref{E:convergence_main_theorem} holds under slightly milder conditions than those needed for \eqref{E:guarantee_separate_estimation} to hold since Condition~\ref{cd8} is a relaxed version of the \emph{beta-min condition} in~\cite{vandegeer2013}. A more detailed discussion about the conditions of Theorem~\ref{thm:l0} is given next.

\subsection{Conditions for Theorem~\ref{thm:l0}}\label{Sss:conditions_for_theorem}

It has been shown in~\cite{RU13} that if a data-generating distribution is faithful with respect to $\G$, then $\G$ must be a minimal-edge I-MAP. 
By enforcing the latter for every true graph, Condition~\ref{cd1} imposes a milder requirement compared to the well-established faithfulness assumption~\cite{spirtes:00}.
Conditions~\ref{cd3}-\ref{cd5} ensure that we avoid overfitting and provide bounds for the noise variances. These are direct adaptations from Conditions~3.1-3.3 in ~\cite{vandegeer2013}.
Condition~\ref{cd6} is required to bound the difference between the sample variances of our observations and the true variances, and is related to Condition~3.4 in~\cite{vandegeer2013} but adapted to our joint inference setting. Notice that Condition~\ref{cd6} is trivially satisfied when $p = \mathcal{O}\left( \frac{n^{1/3}}{K^{7/3} \log n} \right)$.
Condition~\ref{cd9} follows from the bounds for sample variances shown in~\cite{cai2015joint}. Intuitively, we are imposing the natural restriction that the number of DAGs is small compared to the number of nodes $p$ in each DAG and the total number of observations $n$. Moreover, given that our objective is to draw estimation power from the joint inference of multiple graphs, we require that each DAG is associated with a non-vanishing fraction of the total observations.

Condition~\ref{cd7} enforces that, for every permutation $\pi$, the number of edges in the union of all recovered graphs is proportional to the weighted sum of the edges in every graph as $K \to \infty$. In particular, this requires the individual graphs $\G_{0 \pi}^{(k)}$ to be highly overlapping. To see why this is the case, notice that $\sum_{k=1}^K w_k \vert \G_{0 \pi}^{(k)} \vert$ is upper bounded by the maximum number of edges across graphs $\G_{0 \pi}^{(k)}$. Consequently, Condition~\ref{cd7} enforces the number of edges in the union of graphs to be proportional to the number of edges in the single graph with most edges, thus requiring a high level of overlap. Imposing high overlap for all permutations $\pi$ might seem too restrictive in some settings.
Nonetheless, Condition~\ref{cd7} can sometimes be derived from apparently less restrictive conditions as the following example illustrates.

Consider the more relaxed bound $\vert \G_0^\un \vert \leq c_t \sum_{k=1}^K w_k \vert \G_0^{(k)} \vert$, which is equivalent to requiring Condition~\ref{cd7} to hold but \emph{only} for permutations consistent with the true graph $\G_0^\un$. 
In the following example, we show that this might be sufficient for Condition~\ref{cd7} to hold. 
Suppose that $\G_0^\un$ consists of two connected components ${\G'}^\un_0$ and ${\G''}^\un_0$ respectively defined on the subsets of nodes $V_1$ and $V_2$. Moreover, assume that the subDAGs of $\G_0^{(k)}$ over $V_1$ (denoted by ${\G'}^{(k)}_0$) are identical for all $k$. Putting it differently, the differences between the DAGs $\G_0^{(k)}$ are limited to the second connected component. 
In addition, assume that for all possible permutations $\pi_2$ of nodes $V_2$ we have that $\vert {\G''}^\un_{0\pi_2} \vert \leq \vert {\G'}^\un_{0} \vert$. Then, for any permutation $\pi$, where we denote by $\pi_1$ (respectively $\pi_2$) the restriction of $\pi$ to the node set $V_1$ (respectively $V_2$), we have
\begin{align*}
\vert \G_{0\pi}^\un \vert = \vert {\G'}^\un_{0\pi_1} \vert + \vert {\G''}^\un_{0\pi_2} \vert
 \leq \sum_{k=1}^K w_k \vert {\G'}^{(k)}_{0 \pi_1}\vert + \sum_{k=1}^K w_k \vert {\G'}^{(k)}_0\vert \leq 2 \sum_{k=1}^K w_k \vert \G_{0\pi}^{(k)}\vert,
\end{align*}
which shows that Condition~\ref{cd7} is satisfied for $c_t=2$. This example shows that learning the structure of large components that are common across the different DAGs is not affected by the changes in the smaller components of these DAGs. Beyond this example, in Section~\ref{Ss:cd7}, we also provide simulation results to study the strength of Condition~\ref{cd7} for sparse DAG models. Our simulation analysis shows that, when the $\G_0^{(k)}$'s are highly overlapping (recall that this corresponds to a more relaxed scenario than Condition~\ref{cd7} that requires high overlap across $\G_{0\pi}^{(k)}$'s for all $\pi$'s), Condition~\ref{cd7} is naturally satisfied with a reasonably small $c_t$.
Despite the above example as well as the empirical analysis, Condition~\ref{cd7} might still be too restrictive for some applications; we discuss a relaxed requirement and its implications on the consistency guarantees in Section~\ref{Ss:consistency_milder_conditions}. 

Condition~\ref{cd8} requires that, for every permutation $\pi$ and every graph $k$, the value of at least a fixed proportion $(1-\eta_1)$ of the edges in $\G_{0 \pi}^{(k)}$ is above the \emph{`noise level'}, i.e., the lower bound within the indicator function in \eqref{E:noise_level_condition}. 
Intuitively, if the true weight of many edges is close to zero then correct inference of the graphs would be impossible since the true edges would be mistaken with spurious ones. 
Thus, it is expected that the weights of a sufficiently large fraction of the edges have to be sufficiently large. 
Condition~\ref{cd8} is the right formalization of this intuition. 
Moreover, notice that a straightforward replication of the beta-min condition introduced in \cite{vandegeer2013} would have required the \emph{`noise level'} to scale with $\sqrt{\log p / n_k}$, instead of the smaller scaling of $\sqrt{\log p / n}$ required in \eqref{E:noise_level_condition}. 
In this sense, Condition~\ref{cd8} (together with Condition~\ref{cd1}) is a relaxed version of the extension of the beta-min condition to the setting of joint graph estimation.


\begin{remark}[Strength of assumptions]\label{rk:soa}
Requiring strong assumptions for consistent estimation is a common theme in existing methods for causal inference. 
For example, the PC algorithm requires the strong faithfulness assumption~\cite{KB07}, which has been shown to be a very restrictive assumption for high-dimensional causal graphical models~\cite{URBY13}. For a discussion on the comparison between the strong faithfulness assumption and the beta-min condition for estimating a single DAG model, see~\cite[Section~4.3.2]{vandegeer2013}.
In this context, the assumptions presented here are in line with or slight relaxations (Conditions~\ref{cd1} and~\ref{cd8}) of those in state-of-the-art approaches. While it would be interesting in future work to formally compare Conditions~\ref{cd1} and~\ref{cd8} to strong faithfulness, our goal here is not to relax existing assumptions for the estimation of DAG models, but to show that \emph{joint} estimation can result in faster rates than \emph{separate} estimation of multiple DAGs under comparable assumptions.
\end{remark}

\subsection{Consistency under milder conditions}\label{Ss:consistency_milder_conditions}

As previously discussed, in some settings Condition~\ref{cd7} might be too restrictive.
Hence, in this section we present a consistency statement akin to Theorem~\ref{thm:l0} that holds for a milder version of Condition~\ref{cd7}: 

\vspace{0.2cm}
\noindent {\bf Condition~4.7'.} \emph{Let $c_t(\pi)$ be some function of $\pi$ that scales as a constant for permutations consistent with $\G_0^\un$ and scales as $o(K)$ for all other permutations such that $\vert \G_{0 \pi}^{\un} \vert \leq c_t(\pi) \sum_{k=1}^K w_k \vert \G_{0 \pi}^{(k)} \vert$ for all $\pi$.}
\vspace{0.2cm}

Observe that for permutations $\pi$ consistent with the true union graph $\G_0^\un$, Condition~\ref{cd7}' boils down to the previously discussed Condition~\ref{cd7}. However, for all other permutations, $c_t(\pi)$ need not be a constant and is allowed to grow with $K$. Intuitively, for all permutations \emph{not} consistent with $\G_0^\un$ we are \emph{not} requiring a high level of overlap among all the graphs $\G_{0 \pi}^{(k)}$. Nonetheless, since $c_t(\pi) = o(K)$ we \emph{do} require $\G_{0 \pi}^{\un}$ to be \emph{`sparser'} than the extreme case in which all graphs $\G_{0 \pi}^{(k)}$ are disjoint.

In order to account for the fact that $c_t$ depends on the permutation $\pi$ in Condition~\ref{cd7}', we have to modify Condition~\ref{cd8} accordingly, resulting in the following alternative statement.

\vspace{0.2cm}
\noindent {\bf Condition~4.8'.} \emph{Let $C_{\max} := \underset{\pi}{\max}\; c_t(\pi)$, then there exist constants $\eta_0$ and $\eta_1$ such that $0 \leq \eta_1 < 1$, $0 < \eta_0^2 < (1 - \eta_1)$, and 
	\begin{align*}
	\sum_{i,j} \mathbf{1} \left\{ \left| [A_{0\pi}^{(k)}]_{i,j} \right| > \frac{\sqrt{ C_{\max} \log p / n}}{\eta_0} \left( \sqrt{p / \vert \G_0^\un\vert} \vee 1 \right) \right\} \geq (1 - \eta_1)\vert \G_{0 \pi}^{(k)} \vert,
	\end{align*}
	for all permutations $\pi$ and graphs $k$, where $\mathbf{1}\{ \cdot \}$ denotes the indicator function.}
\vspace{0.2cm}

The following consistency result holds for the alternative set of conditions.

\begin{theorem} \label{thm:relax}
	Under Conditions~\ref{cd1}-\ref{cd9}, \ref{cd7}' and \ref{cd8}' and with $\lambda$ such that
	$\lambda^2 \asymp C_{\max} \frac{\log p}{n} \left( \frac{p}{\vert \G_0^\un \vert} \vee 1 \right)$,
	then there exists a constant \,$c>0$\, that depends on $c_s$ such that with probability \,$1 - \exp(-cp)$,\, the solution to~\eqref{obj:jointl0} satisfies that, at least for one $k\in[K]$,
	\begin{align}\label{E:convergence_main_theorem_relaxed}
	\Vert \hat{A}^{(k)} - A_{0 \hat{\pi}}^{(k)} \Vert_F^2 + \Vert \hat{\Omega}^{(k)} - \Omega_{0 \hat{\pi}}^{(k)} \Vert_F^2 = \mathcal{O}\left( \lambda^2 \vert \G_0^{(k)} \vert\right).
	\end{align}
	Furthermore, denoting by $\hat{\G}^{(k)}$ the graph associated with $\hat{A}^{(k)}$ for the $k\in[K]$ satisfying \eqref{E:convergence_main_theorem_relaxed}, we have that 
	\begin{align}\label{E:edges_main_theorem_relaxed}
	\vert \hat{\G}^{(k)} \vert \asymp \vert \G_{0 \hat{\pi}}^{(k)} \vert \asymp \vert \G_0^{(k)} \vert.
	\end{align}
\end{theorem}

The proof is given in Appendix~\ref{sec:relaxproof}. Condition~\ref{cd7}' is milder than Condition~\ref{cd7} and this relaxation entails a corresponding loss in the guarantees of recovery: Comparing \eqref{E:convergence_main_theorem_relaxed} and \eqref{E:edges_main_theorem_relaxed} with \eqref{E:convergence_main_theorem} and \eqref{E:edges_main_theorem} immediately reveals that what could be guaranteed for the ensemble of graphs in Theorem~\ref{thm:l0} can only be guaranteed for a single graph in Theorem~\ref{thm:relax}, thereby explaining the trade-off in relaxing the conditions. 

However, the result in Theorem~\ref{thm:relax} still draws inference power from the joint estimation of multiple graphs since neither \eqref{E:convergence_main_theorem_relaxed} nor \eqref{E:edges_main_theorem_relaxed} can be shown using existing results for separate estimation. 
To be more precise, as discussed in Section~\ref{Sss:conditions_for_theorem}, when performing separate estimation, theoretical guarantees are based on the assumption that at least a fixed proportion of the edge weights are above the \emph{`noise level'}, which scales as $\sqrt{\log p / n_k}$. 
However, Condition~\ref{cd8}' requires the noise level to scale with $\sqrt{C_{\max} \log p /n}$ which, given the fact that $C_{\max} = o(K)$, is not large enough to achieve the guarantee needed for separate estimation.
In addition, the convergence rate of $\Omega(C_{\max} \frac{\log p}{n})$ in \eqref{E:convergence_main_theorem_relaxed} is still faster than the corresponding convergence rate of $\Omega(K \frac{\log p}{n})$ associated with separate estimation [cf. discussion after~\eqref{E:guarantee_separate_estimation}]. 
A potential limitation of Theorem~\ref{thm:relax} is that, since one cannot know which of the $K$ DAGs achieves such $\sqrt{C_{\max} \log p /n}$ rate and which remains at $\sqrt{\log p / n_k}$, the result may be of limited utility for practitioners. 
However, note that the much weaker Condition 4.7' helps to illustrate that, even in such scenario, joint estimation can be helpful compared with separate estimation. 
In addition, it also clarifies which guarantees are lost with respect to the more stringent scenario when Condition 4.7 holds. In this sense, although not fully interpretable, this intermediate case provides an idea of how the guarantees degrade as we start to soften the assumptions.

We end this section with the following remark discussing the consistency guarantees of jointGES.

\begin{remark}[Consistency of jointGES]
In the low-dimensional setting, by choosing $\lambda_1^2 = \sum_{k=1}^K w_k \frac{\log n_k}{2n_k}$, assuming faithfulness and assuming that GES finds the global optimum of~\eqref{E:alg_step_1}, it can be inferred from~\cite{chickering2002optimal} that, in the limit of large data, the first step in Algorithm~\ref{alg:jointl0} is guaranteed to produce a Markov equivalence class (MEC) $\hat{\M}$ that is within the following set of MECs:
\begin{align*}
\M^* := \big\{\M: \;\textrm{there exists a}\; \pi \in \Pi \;\textrm{such that}\; \M = \M(\G_{0\pi}^\un)\big\}
\end{align*}
where
\begin{align*}
\Pi := \big\{\pi : \forall k \in [K], \G_{0\pi}^{(k)} \in \M(\G_0^{(k)})\big\}.
\end{align*}
This allows us to recover $\{\M(\G_0^{(k)})\}_{k=1}^K$ by successively considering all DAGs in $\hat{\M}$ as inputs to the second step of Algorithm~\ref{alg:jointl0}, and selecting the DAG $\hat{\G}^\un \in \hat{\M}$ whose output $\{\hat{\G}^{(k)}\}_{k=1}^K$  from step 2 is the sparsest. Then the MECs $\{\M(\hat{\G}^{(k)})\}_{k=1}^K$ produced from step 2 asymptotically coincide with $\{\M(\G_0^{(k)})\}_{k=1}^K$. Note that no matter which MEC is chosen from the set $\M^*$, by doing edge reductions in step~2, the final result is always guaranteed to asymptotically converge to $\{\M(\G_0^{(k)})\}_{k=1}^K$. In Example~\ref{ex:jointges}, we show how Algorithm~\ref{alg:jointl0} works for a particular 3-node instance. In the high-dimensional setting, where even the global optimum of~\eqref{obj:jointl0} is not guaranteed to recover the true $\G_0^\un$ (cf. Theorems~\ref{thm:l0} and~\ref{thm:relax}), jointGES is in general not consistent.
Recently, Maathuis et al.~\cite{nandy2015high} showed consistency of GES for single DAG estimation in the high-dimensional setting under more restrictive assumptions than the ones considered here.
Although of potential interest, further strengthening the presented conditions to guarantee consistency of jointGES also in the high-dimensional setting is not pursued in the current paper.
\end{remark}

\begin{example}\label{ex:jointges}
We present an example to illustrate how the output from Algorithm~\ref{alg:jointl0} works in the low-dimensional regime. Consider the setting where we have data collected from two causal DAG models on 3 nodes, namely $1 \rightarrow 2$ and $2 \leftarrow 3$. 
In the first step, Algorithm~\ref{alg:jointl0} will produce either the PDAG $1 \rightarrow 2 \leftarrow 3$ or $1 - 2 - 3$. 
Then, no matter which of the two PDAGs is learned in the first step, by taking it to Step~2, it is guaranteed to asymptotically converge to the desired output $1 \to 2$ (or $1 \leftarrow 2$) as well as $2 \to 3$ (or $2 \leftarrow 3$), of which the MECs are $1 - 2$ and $2 - 3$, respectively.
\end{example}

\section{Extension to interventions}\label{sec:interventions}

In this section, we show how our proposed method for joint estimation can be extended to learn DAGs from interventional data. 
It is natural to consider learning from interventional data as a special case of joint estimation since the DAGs associated with interventions are different but closely related.
In this section, we mimic some of the developments of Sections~\ref{sec:framework} and \ref{sec:theory} but specialized for the case of interventional data. More precisely, we first propose an optimization problem akin to \eqref{obj:jointl0} and then state the consistency guarantees in the high-dimensional setting of the associated optimal solution.

Recall from Section~\ref{Ss:collection_dags} that the true adjacency matrix $A_0^{I_k}$ of the SEM associated with an intervention on the nodes $I_k$ is identical to the true adjacency matrix $A_0$ of the non-intervened model except that $[A_0^{I_k}]_{ij} = 0$ for all $j \in I_k$. In this way, our assumption that there exists a common permutation $\pi$ consistent with all DAGs under consideration (cf. Section~\ref{Ss:collection_dags}) is automatically satisfied for interventional data. Additionally, assuming that we observe samples $\hat{X}^{I_k}$ from $K$ different models corresponding to the respective intervention on the nodes in $\{ I_k\}_{k=1}^K$, the knowledge of the intervened nodes can be incorporated into our optimization problem as follows [cf.~\eqref{obj:jointl0}].
\begin{subequations}\label{obj:jointl0_intervened}
	\begin{align} 
	\Big\{\hat{\pi}, & \hat{A}, \hat{\Omega}, \{(\hat{A}^{I_k}, \hat{\Omega}^{I_k}) \}_{k=1}^K \Big\} \nonumber \\
	& = \argmax_{\pi, A, \Omega, \{(A^{I_k}, \Omega^{I_k})\}_{k=1}^K} \;\; && \sum_{k=1}^K w_k \ell_{n_k} (\hat{X}^{I_k}; A^{I_k}, \Omega^{I_k}) - \lambda^2 \left\| A \right\|_0 \label{obj:jointl0_intervened_a} \\
	& \qquad\quad\;\; \textrm{subject to}  && A \in \mathcal{A}_{\pi}, \,\,\,\,  \| A \|_{\infty, 0} \leq d, \,\,\,\, {\Omega} \in \mathcal{D}_+, \label{obj:jointl0_intervened_b} \\ 
	& && A^{I_k}_{ij} = A_{ij} \,\,\, \forall j \not\in I_k, \,\,\,\, A^{I_k}_{ij} = 0 \,\,\, \forall j \in I_k, \label{obj:jointl0_intervened_c} \\
	& && \Omega^{I_k}_{jj} = \Omega_{jj} \,\,\, \forall j \not\in I_k, \,\,\,\,  \Omega^{I_k} \in \mathcal{D}_+. \label{obj:jointl0_intervened_d}
	\end{align}
\end{subequations}
From the solution of \eqref{obj:jointl0_intervened} we obtain an estimate for the non-intervened SEM $(\hat{A}, \hat{\Omega})$ as well as $K$ estimates for the corresponding intervened models $(\hat{A}^{I_k}, \hat{\Omega}^{I_k})$. 
The objective in \eqref{obj:jointl0_intervened_a} is equivalent to that in \eqref{obj:jointl0} where we leverage the fact that the union of all intervened graphs results in the non-intervened one under the implicit assumption that no single node has been intervened in every experiment. Alternatively, if some nodes were intervened in all experiments, objective \eqref{obj:jointl0_intervened_a} would still be valid since enforcing zeros in the unobservable portions of $A$ does not affect the recovery of the intervened adjacency matrices $A^{I_k}$.
The constraints in \eqref{obj:jointl0_intervened_b} impose that $A$ has to be consistent with permutation $\pi$ and with bounded in-degree, and $\Omega$ has to be a valid covariance matrix for uncorrelated noise. Putting it differently, \eqref{obj:jointl0_intervened_b} enforces for the non-intervened SEM what we impose separately for all SEMs in \eqref{obj:jointl0}. The constraints in \eqref{obj:jointl0_intervened_c} impose the known relations between the intervened and the non-intervened adjacency matrices. Finally, \eqref{obj:jointl0_intervened_d} constrains the matrices $\Omega^{I_k}$ to be consistent with the base model on the non-intervened nodes while still being a valid covariance on the intervened ones. 

Even though it might seem that in \eqref{obj:jointl0_intervened} we are estimating $K+1$ SEMs (the base case plus the $K$ intervened ones), from the previous reasoning it follows that the effective number of optimization variables is significantly smaller. To be more specific, for a given $\pi$, once $A$ is fixed then all the adjacency matrices $A^{I_k}$ are completely determined. Moreover, for a fixed $\Omega$, the only freedom in $\Omega^{I_k}$ corresponds to the diagonal entries associated with intervened nodes in $I_k$. In this way, it is expected that for a given number of samples, the joint estimation of $K$ SEMs obtained from interventional data [cf.~\eqref{obj:jointl0_intervened}] outperforms the corresponding estimation from purely observational data [cf.~\eqref{obj:jointl0}]. 

Recalling that we denote by $(A^{I_k}_{0 \hat{\pi}}, \Omega_{0 \hat{\pi}}^{I_k})$ the parameters recovered from the Cholesky decomposition of the true precision matrix $\Theta_0^{I_k}$ under the assumption of consistency with permutation $\hat{\pi}$, the following result holds.

\begin{corollary} \label{cor:intl0}
	If Conditions~\ref{cd1}-\ref{cd8} hold and $\lambda$ is chosen as $\lambda^2 \asymp \frac{\log p}{n} \left( \frac{p}{\vert \G_0 \vert} \vee 1 \right)$, then there exist constants \,$c_1, c_2>0$\, such that with probability \,$1 - c_1 \exp(-c_2p)$,\, the solution to~\eqref{obj:jointl0_intervened} satisfies
	\begin{align}\label{E:cor_1}
	\sum_{k=1}^K w_k \Vert \hat{A}^{I_k} - A^{I_k}_{0 \hat{\pi}} \Vert_F^2 = \mathcal{O}(\lambda^2 \vert \G_0 \vert).
	\end{align}
	Furthermore, denoting by $\hat{\G}$ the graph associated with the recovered adjacency matrix $\hat{A}$ for the non-intervened model, we have that
	\begin{align}\label{E:cor_2}
	\vert \hat{\G} \vert \asymp \vert \G_{0 \hat{\pi}} \vert \asymp \vert \G_0 \vert.
	\end{align}
\end{corollary}


\begin{figure}[b]
	\centering
	\subfigure[]{\includegraphics[scale=1]{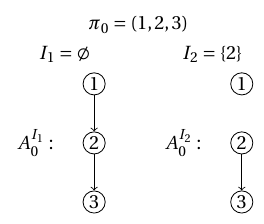}}\qquad \qquad
	\subfigure[]{\includegraphics[scale=1]{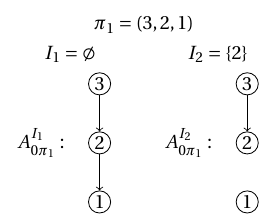}}
	\vspace{-0.2cm}
	\caption{Interventional data can avoid the recovery of spurious permutations. (a) True DAGs to be recovered. (b) DAGs obtained from the Cholesky decomposition consistent with $\pi_1$. The spurious permutation $\pi_1$ does not satisfy \eqref{E:cor_1} for cases where node 2 is intervened.}
	\label{fig:intdag}
	\vspace{-0.3cm}
\end{figure}

The proof is given in Appendix~\ref{sec:corproof}. A quick comparison of Theorem~\ref{thm:l0} and Corollary~\ref{cor:intl0} seems to indicate that the consistency guarantees of observational and interventional data are very similar.
However, recovery from interventional data is strictly better as we argue next.
As discussed after Theorem~\ref{thm:l0}, the results presented do not guarantee that the permutation $\hat{\pi}$ recovered coincides with the true permutation of the nodes. 
In principle, one could recover a spurious permutation $\hat{\pi}$ (different from the true permutation $\pi$) that correctly explains the observed data [cf.~\eqref{E:convergence_main_theorem} and \eqref{E:cor_1}] and leads to sparse graphs [cf.~\eqref{E:edges_main_theorem} and \eqref{E:cor_2}].
However, the more interventions we have, the smaller the set of spurious permutations $\hat{\pi}$ that can be recovered, as we illustrate in the following example.
Figure~\ref{fig:intdag} portrays the existence of a spurious permutation that could be recovered from observational data but cannot be recovered from interventional data.
More precisely, Figure~\ref{fig:intdag}(a) presents the two true DAGs that we aim to recover, where the second one is obtained by intervening on node 2. By contrast, Figure~\ref{fig:intdag}(b) shows the DAGs that are obtained when performing Cholesky decompositions on the true precision matrices under the spurious permutation $\pi_1$. Notice that the sparsity levels of the DAGs in both figures are the same. 
In general, one could recover $\pi_1$ instead of $\pi_0$ from observational data, but one would never recover $\pi_1$ from interventional data. To see this, simply notice from the figure that $[A^{I_2}_{0 \pi_1}]_{32} \neq 0$ whereas for the interventional estimate $[\hat{A}^{I_2}]_{32} = 0$ [cf.~\eqref{obj:jointl0_intervened_c}], thus, the error terms in \eqref{E:cor_1} cannot vanish for $\hat{\pi} = \pi_1$.
This example also indicates that it is preferable to intervene on multiple targets in the same experiment instead of doing interventions one at a time. This observation is in accordance with new genetic perturbation techniques, such as Perturb-seq~\cite{dixit2016perturb}. 

From a practical perspective, the objective in \eqref{obj:jointl0_intervened} corresponds to the same scoring function as GIES~\cite{hauser2012characterization}. Therefore, GIES can be used to obtain an approximate solution to~\eqref{obj:jointl0_intervened}. A simulation study using GIES was performed in~\cite[Section 5.2]{hauser2012characterization} showing that in line with the theoretical results obtained in this section, not only identifiability increases, but also the estimates obtained using interventional data are better than with the same amount of purely observational data.

\section{Experiments}
\label{sec:exp}

In this section, we present numerical experiments on both synthetic (Section~\ref{sec:exp_joint}) and real (Section~\ref{Ss:cancer_data}) data that support our theoretical findings. We also provide an empirical analysis to study the strength of Condition~\ref{cd7} for sparse DAG models in Section~\ref{Ss:cd7}.

\begin{figure}[b!]
	\centering
	\subfigure[Average SHD, $n = 600$]{\includegraphics[width=0.3\textwidth]{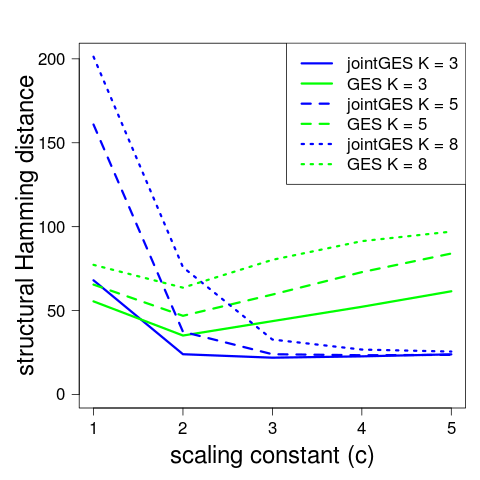}} 
	\subfigure[Average SHD, $n = 900$]{\includegraphics[width=0.3\textwidth]{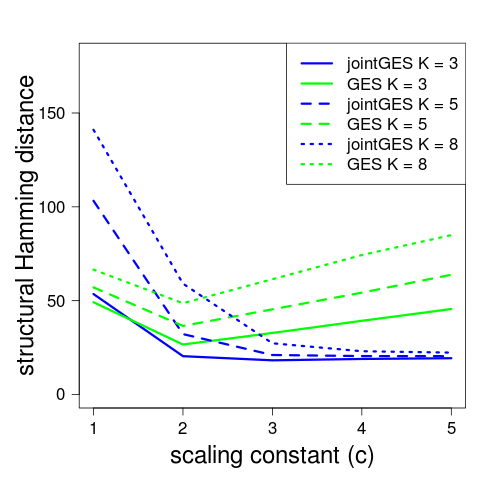}}
	\subfigure[Average SHD, $n = 1200$]{\includegraphics[width=0.3\textwidth]{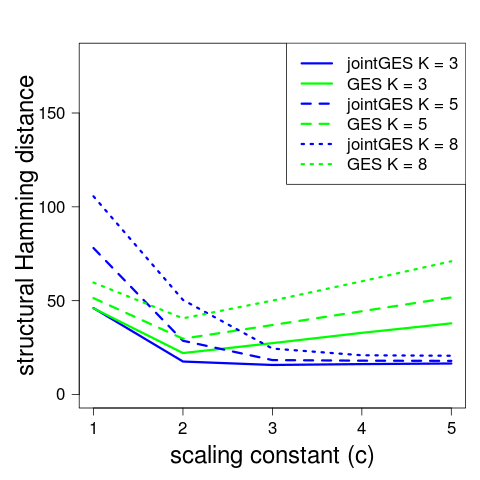}} \\
	\subfigure[Average ROC, $n = 600$]{\includegraphics[width=0.3\textwidth]{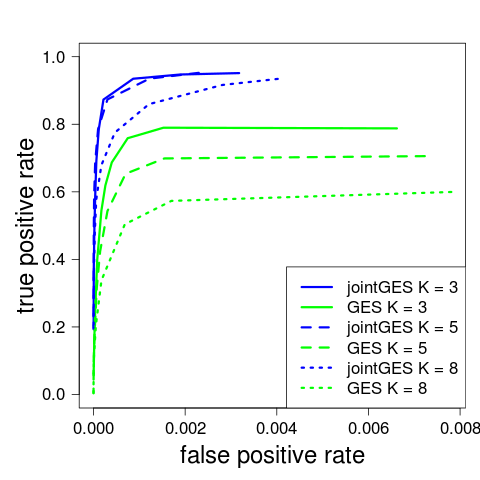}}
	\subfigure[Average ROC, $n = 900$]{\includegraphics[width=0.3\textwidth]{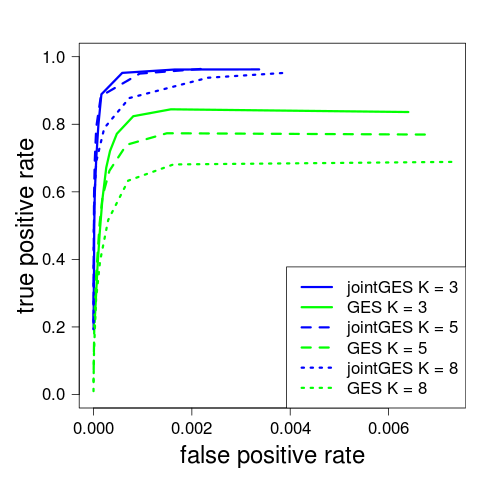}}
	\subfigure[Average ROC, $n = 1200$]{\includegraphics[width=0.3\textwidth]{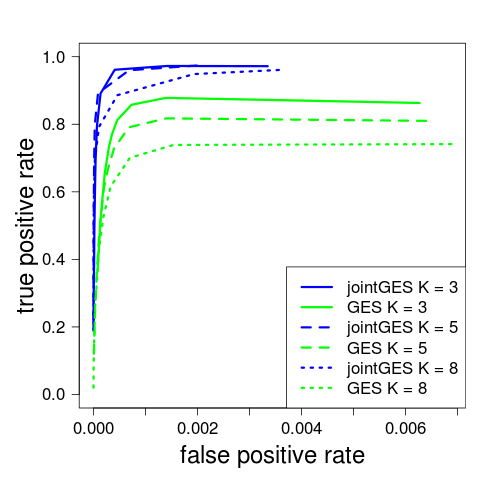}}
	\vspace{-0.2cm}
	\caption{Simulation results when we set the private-to-core edge ratio to $0.3$. (a) - (c) Average SHD as a function of the scaling constant $c$ for joint and separate GES with $n = 600, 900, 1200$ respectively; (d) - (f) Average ROC curve obtained by varying the tuning parameters with $n = 600, 900, 1200$. JointGES consistently achieves a better performance across all settings.
	}
	\label{fig:r3}
	\vspace{-0.2cm}
\end{figure}

\subsection{Performance evaluation of joint causal inference} \label{sec:exp_joint}

We analyze the performance of the joint recovery of $K$ different DAGs where we vary $K \in \{ 3, 5, 8\}$ and $n \in \{600, 900, 1200\}$. For all experiments, we set the number of nodes $p=100$. 
In addition, we selected the number of samples from each DAG to be the same, i.e., $n_1 = \ldots = n_K = n/K$.
For each experiment, the true DAGs were constructed in two steps.
First, we generated a \emph{core} graph that is shared among the $K$ DAGs under consideration.
We did this by generating a random graph from an Erd\H{o}s-R\'enyi model with $100$ edges in expectation, and then oriented the edges according to a random permutation of the nodes.
Then we sampled $e_{\textrm{private}} \in \{30, 60\}$ additional \emph{private} edges uniformly at random. Each such edge was assigned uniformly at random to one of the $K$ DAGs, thereby keeping the total number of private edges across all $K$ DAGs to be $e_{\textrm{private}}$.
This procedure results in the generation of a collection of true underlying DAGs $\G_0^{(1)}, \dots , \G_0^{(K)}$ with a private-to-core edge ratio of $0.3$ and $0.6$ respectively. 
Associated with each DAG, we generated a true adjacency matrix $A^{(k)}_0$ and a true diagonal error covariance matrix $\Omega^{(k)}_0$. 
For the latter, we drew each error variance independently and uniformly from the interval $[1, 2.25]$.
Regarding the adjacency matrices, we drew the edge weights independently and uniformly from $[-1, -0.1] \cup [0.1, 1]$ to ensure that they are bounded away from zero.
Notice that we did not put any constraints on the edge weights that are in the shared core structure for different DAGs: the same edge can change its weight in different DAGs, or even flip sign.
	
We randomly generated $100$ collections of DAGs and data associated with them. 
We then estimated the DAGs from the data via two different methods: a joint estimation procedure using jointGES presented in Algorithm~\ref{alg:jointl0} and a separate estimation procedure using the well-established GES method~\cite{chickering2002optimal}. 

\begin{figure}[t!]
	\centering
	\subfigure[Average SHD, $n = 600$]{\includegraphics[width=0.3\textwidth]{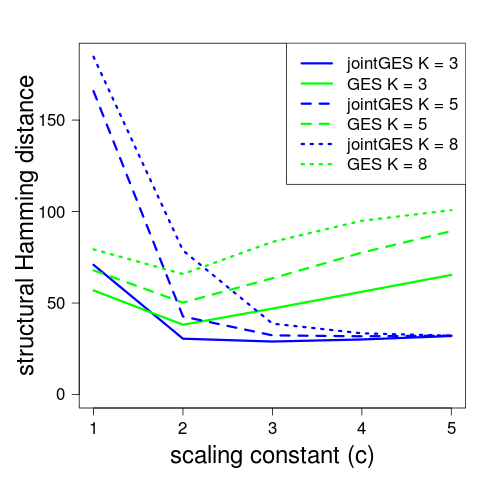}} 
	\subfigure[Average SHD, $n = 900$]{\includegraphics[width=0.3\textwidth]{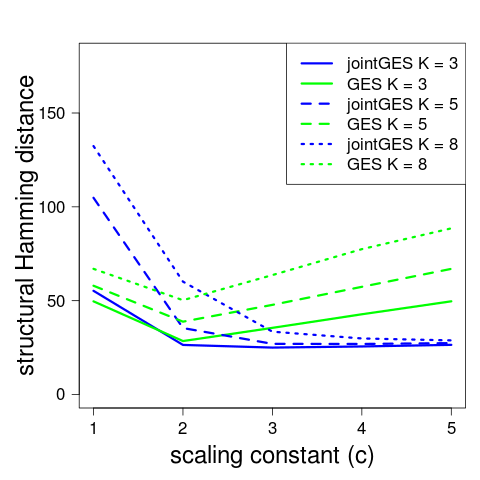}}
	\subfigure[Average SHD, $n = 1200$]{\includegraphics[width=0.3\textwidth]{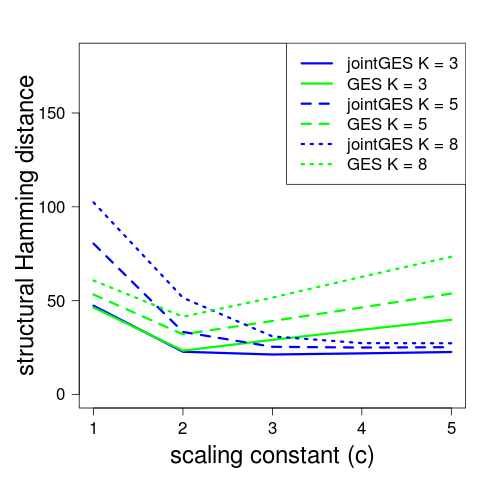}} \\
	\subfigure[Average ROC, $n = 600$]{\includegraphics[width=0.3\textwidth]{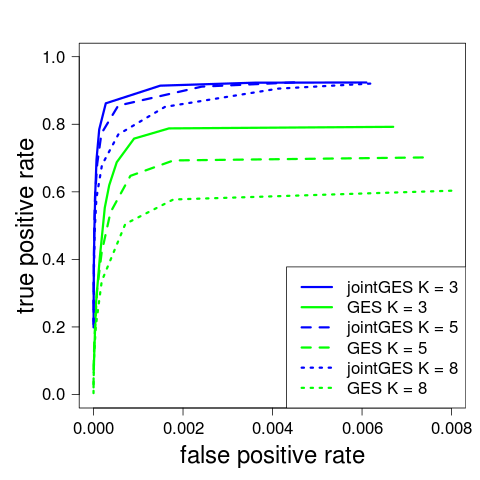}}
	\subfigure[Average ROC, $n = 900$]{\includegraphics[width=0.3\textwidth]{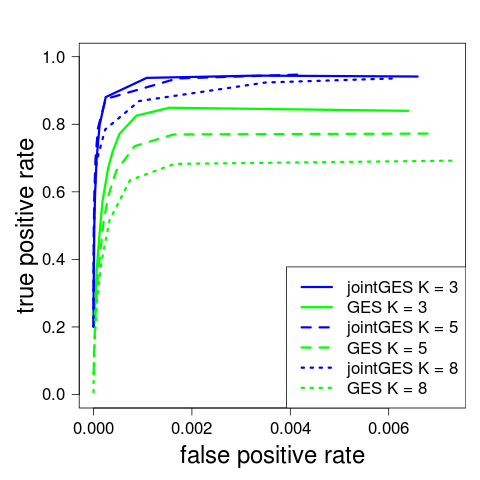}}
	\subfigure[Average ROC, $n = 1200$]{\includegraphics[width=0.3\textwidth]{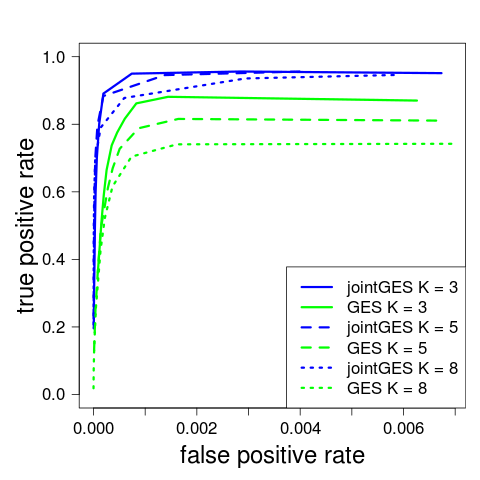}}
	\vspace{-0.2cm}
	\caption{Simulation results when we set the private-to-core edge ratio to $0.6$. (a) - (c) Average SHD as a function of the scaling constant $c$ for joint and separate GES with $n = 600, 900, 1200$ respectively; (d) - (f) Average ROC curve obtained by varying the tuning parameters with $n = 600, 900, 1200$. JointGES consistently achieves a better performance across all settings.
	}
	\label{fig:r6}
	\vspace{-0.2cm}
\end{figure}
	
To assess performance of the two algorithms, we considered two standard measures, namely the structural Hamming distance (SHD)~\cite{tsamardinos2006max} and the receiver operating characteristic (ROC) curve.
SHD is a commonly used metric based on the number of operations needed to transform the estimated DAG into the true one~\cite{KB07, tsamardinos2006max}. 
Hence, a smaller SHD value indicates better performance. 
The ROC curve plots the true positive rate against the false positive rate for different choices of tuning parameters.
The results are shown in Figures~\ref{fig:r3} and~\ref{fig:r6}. Notice that for plotting SHD, we selected the $\ell_0$-penalization parameter $\lambda_1^2 = c \, \frac{\log p}{n}$ with \emph{scaling constant} $c \in \lbrace 1, 2, 3, 4, 5\rbrace$ in both joint and separate estimation and then plotted average SHD as a function of the scaling constant $c$ averaged over the $K$ DAGs to be recovered and the $100$ realizations. The penalization parameter $\lambda_2$ in the second step of the joint estimation procedure was chosen based on $10$-fold cross validation. We plotted the average ROC curve where for each choice of tuning parameter, we computed the true positive and false positive rates by averaging over the $K$ DAGs to be recovered and the $100$ realizations. It is clear from the two figures that in general joint inference achieves better performance, which matches our theoretical results in Section~\ref{sec:theory}.

However, Figures~\ref{fig:r3} (a)-(c) and~\ref{fig:r6} (a)-(c) show also that jointGES performs worse than separate estimation for small scaling constants ($c = 1$). Note that this is in line with our theoretical findings in Theorem~\ref{thm:relax}, which imply that whenever Condition~\ref{cd7} -- which sometimes is a restrictive assumption -- is violated, we need to choose a larger penalization parameter.
	


\subsection{Simulation analysis of Condition~\ref{cd7}} \label{Ss:cd7}

\begin{figure}[b!] 
	\centering
	\subfigure[private-to-core edge ratio $0.3$]{\includegraphics[scale=0.31]{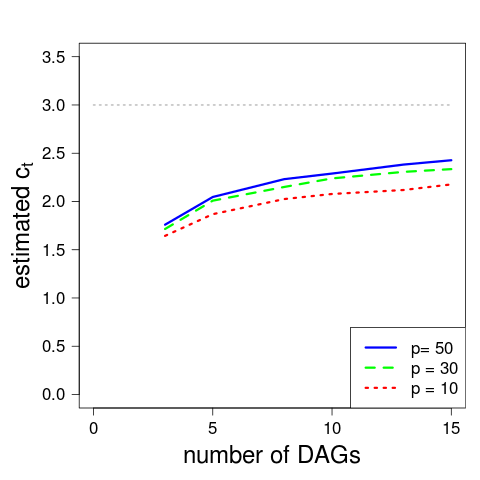}} \qquad\qquad
	\subfigure[private-to-core edge ratio $0.6$]{\includegraphics[scale=0.31]{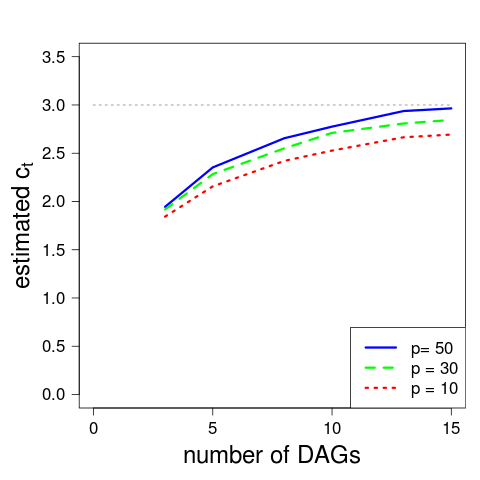}}
	\caption{Averaged $c_t$ across $100$ realizations. (a) is the curve when the private-to-core edge ratio is chosen as $0.3$; (b) corresponds to the curve with private-to-core edge ratio of $0.6$.} \label{fig:cd7}
\end{figure}

In this section we provide simulation results to empirically study the strength of Condition~\ref{cd7}. We follow the same procedure as in Section~\ref{sec:exp_joint} to generate a collection of DAGs, except that we set $p \in \{ 10, 30, 50\}$ and  $K \in \{ 3, 5, 8, 10, 13, 15\}$. Then for each randomly generated $\{\G_0^{(k)}\}_{k=1}^K$, we randomly select $10000$ permutations and estimate the corresponding $c_t$ by  $c_t := \max_{\pi \in \Pi} |\G_{0\pi}^\un| / \big(\frac{1}{K}\sum_{k=1}^K |\G_{0\pi}^{(k)}|\big)$.

In Figure~\ref{fig:cd7} we present the estimated value of $c_t$ as a function of $K$ for $p = 10, 30, 50$. Note that for each setting, we plotted the estimated $c_t$ by averaging over $100$ realizations. The figure shows that if the private-to-core edge ratio is below $0.6$, Condition~\ref{cd7} holds with $c_t \leq 3$, even for very large $K = 15$. This implies that when the true DAGs $\{\G_0^{(k)}\}_{k=1}^K$ are highly overlapping, Condition~\ref{cd7} can usually be satisfied with a reasonably small constant $c_t$. This is in line with the results from Section~\ref{sec:exp_joint}, where we showed that \texttt{jointGES} significantly outperforms the separate estimation approaches.

\subsection{Gene regulatory networks of ovarian cancer subtypes}\label{Ss:cancer_data}

To assess the performance of the proposed joint $\ell_0$-penalized maximum likelihood method on real data, we analyzed gene expression microarray data of patients with ovarian cancer~\cite{tothill2008novel}. 
Patients were divided into six subtypes of ovarian cancer, labeled as C1-C6, where C1 is characterized by significant differential expression of genes associated with stromal and immune cell types and with a lower survival rate as compared to the other 5 subtypes. We hence grouped the subtypes C2-C6 together and our goal was to infer the differences in terms of gene regulatory networks in ovarian cancer that could explain the different survival rates. 
The gene expression data in~\cite{tothill2008novel} includes the expression profile of $n=83$ patients with C1 subtype and $n=168$ patients with other subtypes. We implemented our jointGES algorithm (Algorithm~\ref{alg:jointl0}) to jointly learn \emph{two} gene regulatory networks: one corresponding to the C1 subtype $\G_{\mathrm{C1}}$ and another corresponding to the other five subtypes together $\G_{\mathrm{C2-6}}$. 
In addition, as in~\cite{cai2015joint}, we focused on a particular pathway, namely the \emph{apoptosis} pathway. Using the KEGG database~\cite{kanehisa2011kegg,ogata1999kegg} we selected the genes in this pathway that were associated with at most two microarray probes, resulting in a total of $p=76$ genes.

Table~\ref{tab:edges} lists the number of edges discovered by jointGES as well as for two separate estimation methods, namely using the GES~\cite{chickering2002optimal} and PC~\cite{spirtes:00} algorithms. All three methods were combined with stability selection~\cite{Meinshausen2010} in order to increase robustness of the output and provide a fair comparison. 
As expected, the two graphs inferred using jointGES share a significant proportion of edges, whereas the overlap is markedly smaller for the two separate estimation methods. Interestingly, under all estimation methods the network for the C1 subtype contains fewer edges than the network of the other subtypes, thereby suggesting that $\G_{\mathrm{C1}}$ could lack some important links that are associated with patient survival.

To obtain more insights into the relevance of the obtained networks, we analyzed the inferred hub nodes in the three networks. For our analysis we defined as hub nodes those nodes for which the sum of the in- and out-degree was larger than some threshold $T$ in the union of the two DAGs, where $T$ was chosen such that there are at most $5$ hub nodes discovered by each method. For jointGES, this union is given by $\hat{\G}^\un$, the graph identified in the first step of Algorithm~\ref{alg:jointl0}. The hub nodes identified by jointGES are CAPN1, CTSD, LMNB1, CSF2RB, BIRC3. Among these, CAPN1~\cite{gau2015brca1}, CTSD~\cite{SSB03}, LMNB1~\cite{santin2004gene}, and BIRC3~\cite{Jonsson2014} have been reported as being central to ovarian cancer in the existing literature. In addition, CSF2RB was also discovered by joint estimation of undirected graphical models on this data set~\cite{cai2015joint}. The hub nodes discovered by GES are ATF4, BIRC2, CSF2RB, TUBA1C, MAPK3, while PC only discovered the hub node CSF2RB. While we were not able to validate the relevance of any of these genes for ovarian cancer in the literature, interestingly, CSF2RB has been identified as a hub node by all methods, thereby suggesting this gene as an interesting candidate for future genetic intervention experiments.


\begin{table}
	\centering
	\begin{tabular}{c | c c c}
		\hline
		Method & $|\G_{\mathrm{C1}}|$ & $|\G_{\mathrm{C2-C6}}|$ & $|\G_{\mathrm{C1}} \cap \G_{\mathrm{C2-C6}}|$ \\
		\hline
		JointGES & 50 & 73 & 48 \\
		GES & 68 & 101 & 32 \\
		PC & 14 & 30 & 9 \\
		\hline
	\end{tabular}
	\caption{Number of edges in the DAGs estimated by different methods for the gene regulatory network of subtype C1 ($|\G_{\mathrm{C1}}|$) and all other subtypes ($|\G_{\mathrm{C2-6}}|$). The last column shows the number of edges shared between both inferred graphs.}
	\label{tab:edges}
\end{table}

\section{Discussion}
\label{sec:disscusion}

In this paper we presented jointGES, an algorithm for the joint estimation of multiple \emph{related} DAG models from independent realizations. Joint estimation is of particular interest in applications where data is collected not from a single DAG, but rather multiple related DAGs, such as gene expression data from different tissues, cell types or from different interventional experiments. JointGES first estimates the union of DAGs $\hat{\G}^\un$ by applying a greedy search to approximate the joint $\ell_0$-penalized maximum likelihood estimator and then it uses variable selection to discover each DAG as a subDAG of $\hat{\G}^\un$. 
From an algorithmic perspective, jointGES is to the best of our knowledge the first method to jointly estimate related DAG models in the high-dimensional setting.
From a theoretical perspective, we presented consistency guarantees on the joint $\ell_0$-penalized maximum likelihood estimator, and showed that the accuracy bound scales with the total number of samples, rather than the number of samples from each DAG that would be achieved by separately estimating each DAG. 
As a corollary to this result, we obtained consistency guarantees for $\ell_0$-penalized maximum likelihood estimation of a causal graph from a mix of observational and interventional data.
Finally, we validated our results via numerical experiments on simulated and real-world data, showing that the proposed jointGES algorithm for joint inference outperforms separate-inference approaches using well-established algorithms such as PC and GES.

The present work serves as a platform for the potential development of multiple future directions. These directions include: i) relaxing the assumption that all DAGs must be consistent with the same underlying permutation; ii) extending jointGES to the setting where the samples come from $K$ related DAGs but it is unknown a priori which particular DAG each sample comes from; this is for example of interest in the analysis of gene expression data from tumors or tissues that consist of a mix of cell types; iii) extending jointGES to allow for latent confounders.

\section*{Acknowledgments}
The authors thank two anonymous reviewers for their thoughtful comments, which helped improve our paper. Santiago Segarra was supported by an MIT IDSS seed grant. Caroline Uhler was partially supported by NSF (DMS-1651995), ONR (N00014-17-1-2147 and N00014-18-1-2765), IBM, a Sloan Fellowship and a Simons Investigator Award.

\begin{appendix}

	\section{Theoretical analysis of statistical consistency results} \label{sec:proof}	
	In the following, we develop the proofs of Theorems~\ref{thm:l0} and~\ref{thm:relax}. 
	To facilitate understanding, we first provide a high-level explanation of the rationale behind the proof.
	If we have data generated from a single DAG and we are given a permutation $\pi$ consistent with the true DAG a priori, then we can estimate $(\hat{A}, \hat{\Omega})$ by performing $p$ regressions as explained in Section~\ref{Ss:dags_and_sems}.
	By contrast, when the permutation is unknown and we need to consider all the possible permutations, the total number of regressions to run increases to $p \cdot p !$. However, these regressions are not independent and, intuitively, by bounding the noise level of a subset of these regressions, we can derive bounds for the noise on the other ones.
	We characterize the `noise level' of these regressions by analyzing the asymptotic properties of three random events. More precisely, whenever these events hold -- and we show that they hold with high probability --, the noise is small enough so that the error of the $\ell_0$-penalized maximum likelihood estimator can also be bounded. Finally, we use this upper bound in the error to show that the recovered graph converges to a minimal I-MAP that is as sparse as the true DAG.

	The remainder of the appendix is organized as follows. In Section~\ref{Ss:random_events} we define the three random events previously mentioned and show that each of them holds with high probability. Section~\ref{sec:l0proof} then leverages the definition of these events to prove Theorem~\ref{thm:l0}, our main result. In Section~\ref{sec:relaxproof} we prove Theorem~\ref{thm:relax}, which relaxes some of the conditions of  Theorem~\ref{thm:l0}, but uses similar proof techniques. Finally, Section~\ref{sec:corproof} fleshes out the proof of Corollary~\ref{cor:intl0}, our result applicable to the setting for interventional data.
	
	Throughout the appendix, we use the following notation. Let $\hat{a}_{j}$ denote the $j$-th column of $\hat{A}$ and $\hat{\omega}_j$ denote the $j$-th diagonal entry of $\hat{\Omega}$. Also, denote by $a_{0j\pi}$ and $\omega_{0j\pi}$ the $j$-th column of $A_{0\pi}$ and the $j$-th diagonal entry of $\Omega_{0\pi}$, respectively. 
	
	\subsection{Random events}\label{Ss:random_events}
	
	As in \cite{vandegeer2013}, our proofs of Theorems~\ref{thm:l0} and~\ref{thm:relax} are based on a set of random events. 
	However, the events considered here differ from those in \cite{vandegeer2013} since, as explained in Section~\ref{sec:maintheory}, a naive application of the guarantees in \cite{vandegeer2013} to the joint estimation scenario does not achieve the desired learning rates [cf.~discussion after~\eqref{E:guarantee_separate_estimation}].
	Intuitively, the rate gain achieved here comes from the assumption that all DAGs are consistent with a permutation, allowing us to effectively reduce the size of the search space.
	
	In our proofs we consider three random events $\Ev_1$, $\Ev_2$, and $\Ev_3$ that are respectively stated -- along with proofs showing that they hold with high probability -- in Sections \ref{Sss:random_event_1}, \ref{Sss:random_event_2}, and \ref{Sss:random_event_3}.
	
	\subsubsection{Random event $\Ev_1$}\label{Sss:random_event_1}
	
	Let $\epsilon_{j\pi}^{(k)} \in \R^n$ denote the residual when regressing $X_j^{(k)}$ on $X_S^{(k)}$ with $S = \lbrace i \,|\, \pi(i) < \pi(j)\rbrace$, i.e., $\epsilon_{j\pi}^{(k)} := X_j^{(k)} - {X^{(k)}}^T a_{0j\pi}^{(k)}$. 
	Similarly, let $\hat{\epsilon}_{j\pi}^{(k)} \in \R^n$ denote the regression residual from the sampled data, i.e., $\hat{\epsilon}_{j\pi}^{(k)}  := \hat{X}_j^{(k)} - \hat{X}^{(k)} a_{0j\pi}^{(k)}$. 
	Consider a generic set $\{ A^{(k)}\}_{k=1}^K$ of adjacency matrices consistent with a given permutation $\pi$, where the columns of $A^{(k)}$ are denoted by $A^{(k)}:= (a_1^{(k)}, \ldots, a_p^{(k)})$, and let $\G^\un$ denote the union of the support of $A^{(1)}, \ldots, A^{(K)}$. 
	Then, event $\Ev_1$ is defined as
	\begin{align}\label{E:event_E_1}
	\begin{split}
	\Ev_1 := & \Bigg\{ 2 \sum_{j=1}^p \sum_{k=1}^K \frac{w_k}{n_k} \left\vert
	\hat{\epsilon}_{j\pi}^{(k)T} \hat{X}^{(k)} (a_j^{(k)} - a_{0j\pi}^{(k)}) \right\vert \\ & \leq \delta_1 \sum_{j=1}^p \sum_{k=1}^K \frac{w_k}{n_k} \left\Vert \hat{X}^{(k)} (a_j^{(k)} -  a_{0j\pi}^{(k)}) \right\Vert_2^2 + \lambda_1^2 \left( \vert \G^\un \vert + \vert \G_{0\pi}^\un\vert \right) / \delta_1, \\  & \quad  \forall \, \text{permutations}\, \pi\ \text{, and}\ \forall \, \{ A^{(k)}\}_{k=1}^K \, \text{consistent with}\, \pi \Big\rbrace,
	\end{split}
	\end{align}
	for some constant $\delta_1 > 0$ and some $\lambda_1 \asymp \sqrt{(\log p)/n}$. 
	On random event $\Ev_1$ a uniform inequality holds across all $K$ DAGs for the \emph{sample correlation} between the regression residual $\epsilon_{j \pi}^{(k)}$ and any random variable spanned by the random vector $X_S^{(k)}$, i.e., ${X^{(k)}}^T v$ for any $v \in \R^p$ with $v_i = 0$ for all $i \not\in S$. 
	Notice that for convenience for further steps of the analysis, this generic vector $v$ is written as $a_j^{(k)} - a_{0j\pi}^{(k)}$ in \eqref{E:event_E_1}.
	Furthermore, for simplicity in the rest of this appendix, we denominate the space spanned by $X_S^{(k)}$ as the {\it projection space} of $\epsilon_{j\pi}^{(k)}$.
	Intuitively, one could foresee $\Ev_1$ holding since the expected correlation between the regression residual $\epsilon_{j\pi}^{(k)}$ and $X_S$ is equal to zero, and therefore the sample correlation can be upper bounded by a sum of terms that converge to zero as $n \to \infty$ as in~\eqref{E:event_E_1}.

	We now show that random event $\Ev_1$ holds with high probability, a result stated in Theorem~\ref{thm:lambda1}. 
	Essential towards proving this result is the observation that, since the random variable ${X^{(k)}}^T (a_j^{(k)} - a_{0j\pi}^{(k)})$ lies within the projection space of $\epsilon_{j\pi}^{(k)}$, these two random variables are independent.
	We can therefore deal with the randomness in $\hat{\epsilon}_{j\pi}^{(k)}$ and $\hat{X}_S^{(k)}$ separately, one at a time.
	To formally leverage this observation, we rely on Lemma~7.4 of~\cite{vandegeer2013}, stated next for completeness.
	
	\begin{lemma}
		\cite[Lemma 7.4]{vandegeer2013}
		\label{lm:lambda1}
		Let $Z$ be a fixed $n \times m$ matrix and $e_1, \cdots, e_n$ be independent $\N(0, \sigma_e^2)$-distributed random variables. Then for all $t > 0$
		\begin{align*}
		\PP \left(\underset{\Vert Z a\Vert_2^2 / n \leq 1}{\sup} | e^T Z a | / n \geq \sigma_e (\sqrt{2 m / n} + \sqrt{2t / n})\right) \leq \exp(-t).
		\end{align*}
	\end{lemma}
	
	Based on the above lemma and recalling that $\mathcal{A}_\pi$ denotes the set of adjacency matrices consistent with a given permutation $\pi$, we can show the following result.
	
	\begin{theorem} \label{thm:lambda1}
		Assume that Conditions~\ref{cd3} and~\ref{cd9} hold, then for all $t > 0$ and all $\delta_1 > 0$,
		\begin{align*}
		\PP \, \Bigg( \max_{\pi}  \sup_{\{A^{(k)}\}_{k=1}^K \in \mathcal{A}_\pi} & 2 \sum_{j=1}^p \sum_{k=1}^K \frac{w_k}{n_k} \left\vert \hat{\epsilon}_{j\pi}^{(k)T} \hat{X}^{(k)} (a_j^{(k)} - a_{0j\pi}^{(k)}) \right\vert \\ 
		- \delta_1 \sum_{j=1}^p \sum_{k=1}^K & \frac{w_k}{n_k} \left\Vert \hat{X}^{(k)} (a_j^{(k)} -  a_{0j\pi}^{(k)}) \right\Vert_2^2 \\
		& \left. \geq \frac{16\sigma_0^2 (t + 2 \log p)(\vert \G^\un \vert + \vert \G_{0\pi}^\un\vert)}{n \delta_1} \right) \leq \exp(-t).
		\end{align*}
	\end{theorem}
	
	
	
	\begin{proof}
		Let $\hat{\epsilon}_{j\pi}$ and $a_{0j\pi}$ be the concatenated vectors
		$\hat{\epsilon}_{j\pi}:= (\hat{\epsilon}_{j\pi}^{(1)T}, \ldots, \hat{\epsilon}_{j\pi}^{(K)T})^T $ and
		$a_{0j\pi} := (a_{0j\pi}^{(1)T}, \ldots, a_{0j\pi}^{(K)T} )^T$.
		Also, define the block diagonal matrix $$\hat{X} := \text{diag}(\hat{X}^{(1)}, \cdots, \hat{X}^{(K)}).$$
		We denote by $\mathcal{A}_{j\pi} \subset \R^{pK}$ the set containing all vectors that can be formed by vertically concatenating the $j$th columns $a_j^{(k)}$ for all $k$ and satisfy
		\begin{align*}
		\mathcal{A}_{j\pi} := \left\lbrace a_j \in \R^{pK} \mid \forall i, \textrm{if}\ \exists k\ \textrm{such that}\ a_{i,j}^{(k)} \neq 0,\ \textrm{then}\ X_i^{(k)} \independent \epsilon_{j \pi}^{(k)} \ \textrm{for all $k$} \right\rbrace.
		\end{align*}
		Based on this, and recalling that ${\pa}_{j}(\cdot)$ denotes the set of parent nodes of $j$ in the argument graph, we define the random event $\B_{j\pi}$ as
		\begin{align}\label{E:event_B_j}
		\B_{j\pi} := & \left\lbrace \exists a_j \in \mathcal{A}_{j\pi} : \sup_{\Vert \hat{X}(a_j - a_{0j\pi}) \Vert_2^2 / n \leq 1} \left\vert \hat{\epsilon}_{j\pi}^T \hat{X} (a_j - a_{0j\pi}) \right\vert / n \right. \\
		&\qquad \geq \sigma_0 \left( \sqrt{\frac{2 K (|{\pa}_{j}(\G^{\un})| + |{\pa}_{j}(\G_{0\pi}^{\un})|)}{n}} \right. \nonumber\\
		& \qquad\qquad + \left.\left.\sqrt{\frac{2 (t + |{\pa}_{j}(\G_{0\pi}^{\un})| \log p + 2 \log p)}{n}} \right) \right\rbrace. \nonumber
		\end{align}
		Combining the facts that: i) $a_j - a_{0j\pi}$ may have at most $K (|{\pa}_{j}(\G^{\un})| + |{\pa}_{j}(\G_{0\pi}^{\un})|)$ non-zero entries, and 
		ii) the variance of each element of $\hat{\epsilon}_{j\pi}$ is upper bounded by $\sigma_0^2$ (cf.~Condition~\ref{cd3}), we may apply Lemma~\ref{lm:lambda1} to show that
		\begin{align*}
		\PP(\B_{j\pi}) \leq \exp\left(-t - |{\pa}_{j}(\G_{0\pi}^{\un})| \log p - 2 \log p\right).
		\end{align*}
		As can be seen from \eqref{E:event_B_j}, event $\B_{j\pi}$ depends exclusively on the set of parents of node $j$ in $\G_{0\pi}^\un$. 
		Putting it differently, if for two permutations $\pi_1$ and $\pi_2$ node $j$ has the same set of parents in $\G_{0\pi_1}^\un$ and $\G_{0\pi_2}^\un$, then the random events $\B_{j\pi_1}$ and $B_{j\pi_2}$ coincide, since $\mathcal{A}_{j\pi}$, $a_{0j\pi}$ and $\hat{\epsilon}_{j\pi}$ would all be the same for $\pi \in \{ \pi_1, \pi_2 \}$.
		If we denote by $\Pi_j(m)$ the set of permutations where node $j$ has exactly $m$ parents in $\G_{0\pi}^\un$, then there are at most ${p \choose m}$ unique events $\B_{j\pi}$ for all $\pi \in \Pi_j(m)$. We therefore obtain that
		\begin{align*}
		\PP \left( \bigcup_{\pi \in \Pi_j(m)} \B_{j\pi} \right) \leq {p \choose m} \exp\left(-t - m \log p - 2 \log p\right) \leq \exp\left(-t - 2 \log p\right).
		\end{align*}
		Applying a union bound on the events $\B_{j\pi}$ across all nodes $j$ and permutations $\pi$ yields that
		\begin{align}\label{E:union_bound_B_j}
		\PP \left(\bigcup_{j} \bigcup_{\pi} \B_{j\pi}\right)
		\leq \sum_{j=1}^p \sum_{m=1}^{p-1} \PP\left(\bigcup_{\pi \in \Pi_j(m)} \B_{j\pi}\right)
		\leq p^2 \exp\left(-t - 2 \log p\right) = \exp(-t).
		\end{align}
		Combining \eqref{E:union_bound_B_j} and \eqref{E:event_B_j} it follows that with probability at least $1 - \exp(-t)$, for all $j$, $\pi$ and all $a_j \in \mathcal{A}_{j\pi}$,
		\begin{align*}
		\frac{\vert \hat{\epsilon}_{j\pi}^T \hat{X} (a_j - a_{0j\pi}) \vert}{\Vert \hat{X} (a_j - a_{0j\pi}) \Vert_2 } \leq \sigma_0 & \left( \sqrt{2 K (|{\pa}_{j}(\G^{\un})| + |{\pa}_{j}(\G_{0\pi}^{\un})|)}\right. \\
		& \qquad\qquad+ \left.\sqrt{2 (t + |{\pa}_{j}(\G_{0\pi}^{\un})| \log p + 2 \log p)} \right).
		\end{align*}
		Based on the collection of adjacency matrices $\{A^{(k)}\}_{k=1}^K \in \mathcal{A}_\pi$ we define another collection
		$\{A'^{(k)}\}_{k=1}^K$ where each column ${a'}^{(k)}_j$ is given by
		\begin{align*}
		{a'}^{(k)}_j = \left\lbrace 
		\begin{array}{ll}
		a^{(k)}_j & \mathrm{if} \,\, \hat{\epsilon}^{(k)T}_{j\pi} \hat{X}^{(k)} (a_j^{(k)} - a_{0j\pi}^{(k)}) \geq 0, \\
		2a_{0j\pi}^{(k)} - a_j^{(k)} \,\,\, & \mathrm{otherwise}.
		\end{array}
		\right.
		\end{align*}
		Notice that the positions of the non-zero entries in ${a'}_j^{(k)} - a_{0j\pi}^{(k)}$  coincide with those in $a_j^{(k)} - a_{0j\pi}^{(k)}$. By also using the fact that
		\begin{align*}
		& \| \hat{X} (a_j - a_{0j\pi}) \|_2^2 = \sum_{k=1}^K \| \hat{X}^{(k)} (a_j^{(k)} - a_{0j\pi}^{(k)}) \|_2^2 \\
		& \qquad= \sum_{k=1}^K \| \hat{X}^{(k)} ({a'}^{(k)}_j - a_{0j\pi}^{(k)}) \|_2^2 = \| \hat{X} (a_j' - a_{0j\pi}) \|_2^2,
		\end{align*}
		we have that for all $j$ and $\pi$ with probability at least $1 - \exp(-t)$, it holds that
		\begin{align*}
		& \frac{\sum_{k=1}^K \vert \hat{\epsilon}^{(k)T}_{j\pi} \hat{X}^{(k)} (a_j^{(k)} - a_{0j\pi}^{(k)}) \vert }{\Vert \hat{X} (a_j - a_{0j\pi}) \Vert_2} = \frac{\sum_{k=1}^K \vert \hat{\epsilon}^{(k)T}_{j\pi} X^{(k)} ({a'}_j^{(k)} - a_{0j\pi}^{(k)}) \vert}{\Vert \hat{X} (a_j' - a_{0j\pi}) \Vert_2 } \\
		& \qquad \qquad = \frac{\vert \hat{\epsilon}_{j\pi}^T \hat{X} (a_j' - a_{0j\pi}) \vert}{{\Vert \hat{X} (a_j' - a_{0j\pi}) \Vert_2 }} \leq \sigma_0 \left( \sqrt{2 K (|{\pa}_{j}(\G^{\un})| + |{\pa}_{j}(\G_{0\pi}^{\un})|)} \right.\\
		& \qquad \qquad\qquad\qquad\qquad\qquad\qquad\quad + \left.\sqrt{2 (t + |{\pa}_{j}(\G_{0\pi}^{\un})| \log p + 2 \log p)} \right).\nonumber
		\end{align*}
In the above expression we may use that $ab \leq \delta_1 a^2 + b^2 / \delta_1$ for all $\delta_1, a, b > 0$  in order to obtain
\begin{align}\label{eq:l1bnd_2}
&2 \vert \hat{\epsilon}_{j\pi}^T \hat{X} (a_j' - a_{0j\pi}) \vert \leq \delta_1 \Vert  \hat{X} (a_j' - a_{0j\pi}) \Vert_2^2 \, + \\
&\qquad \qquad 4 \sigma_0^2 \left( \sqrt{2 K (|{\pa}_{j}(\G^{\un})| + |{\pa}_{j}(\G_{0\pi}^{\un})|)}\right.\nonumber \\
&\qquad\qquad\qquad + \left.\sqrt{2 (t + |{\pa}_{j}(\G_{0\pi}^{\un})| \log p + 2 \log p)} \right)^2 / \delta_1.\nonumber
\end{align}
By combining this with the fact that $\forall a, b > 0$, $(a + b)^2 \leq 2(a^2 + b^2)$ and the fact that $K = o(\log p)$ from Condition~\ref{cd9}, we get that
\begin{align}\label{eq:l1bnd_3}
& \left( \sqrt{2 K (|{\pa}_{j}(\G^{\un})| + |{\pa}_{j}(\G_{0\pi}^{\un})|)} + \sqrt{2 (t + |{\pa}_{j}(\G_{0\pi}^{\un})| \log p + 2 \log p)} \right)^2 \nonumber\\
& \hspace{4cm} \leq 4 (t + 2 \log p) (|{\pa}_{j}(\G^{\un})| + |{\pa}_{j}(\G_{0\pi}^{\un})|).
\end{align}
By replacing \eqref{eq:l1bnd_3} back into~\eqref{eq:l1bnd_2}, we obtain that
\begin{align*}
2 \vert \hat{\epsilon}_{j\pi}^T \hat{X} (a_j' - a_{0j\pi}) \vert / n \leq & \delta_1 \Vert  \hat{X} (a_j' - a_{0j\pi}) \Vert_2^2 / n \\
& \quad + \frac{16 \sigma_0^2 (t + 2 \log p)(|{\pa}_{j}(\G^{\un})| + |{\pa}_{j}(\G_{0\pi}^{\un})|)}{n \delta_1}.
\end{align*}
Rewriting the absolute value in the left-hand side as the sum of the corresponding $K$ absolute values and adding the above inequality for $j = 1, \ldots, p$ we get that, with probability at least $1 - \exp(-t)$,
\begin{align*}
\begin{split}
& 2 \sum_{j=1}^p \sum_{k=1}^K w_k \vert \hat{\epsilon}_{j \pi}^{(k)T} \hat{X}^{(k)} (a_j^{(k)} - {a_0}_{j \pi}^{(k)}) \vert / n_k  \\
& \quad\leq \!\sum_{j=1}^p \!\left( \!\delta_1 \Vert \hat{X} (a_j - a_{0j\pi}) \Vert_2^2 / n \!+\! \frac{16\sigma_0^2 (t \!+\! 2 \log p)(|{\pa}_{j}(\G^{\un})| \!+\! |{\pa}_{j}(\G_{0\pi}^{\un})|)}{n \delta_1} \!\right).
\end{split}
\end{align*}
Finally, by noticing that ${\Vert \hat{X} (a_j - a_{0j\pi}) \Vert_2^2}/{n} = \sum_{k=1}^K {w_k} \Vert \hat{X}^{(k)} (a_j^{(k)} - a_{0j\pi}^{(k)}) \Vert_2^2 / {n_k}$ we recover the statement of the theorem, thereby concluding the proof.
\end{proof}

\subsubsection{Random event $\Ev_2$}\label{Sss:random_event_2}
	
	Let $\omega_{0j\pi}^{(k)}$ denote the $j$-th diagonal entry of $\Omega_{0\pi}^{(k)}$, then $\Ev_2$ holds whenever the empirical variances of all $\epsilon_{j\pi}^{(k)}$, i.e., $\Vert \hat{\epsilon}_{j\pi}^{(k)} \Vert_2^2 / {n_k}$ are close to the true variances $\omega_{0j\pi}^{(k)}$, where
	\begin{align}\label{E:event_2}
	\Ev_2 := \left\lbrace \sum_{j=1}^p \sum_{k=1}^K w_k \left( \frac{\Vert \hat{\epsilon}_{j\pi}^{(k)} \Vert_2^2 / n_k - \omega_{0j\pi}^{(k)}}{\omega_{0j\pi}^{(k)}} \right)^2 \leq 4 \, \lambda_2^2 \left(p + \vert \G_{0\pi}^\un\vert \right)\right\rbrace,
	\end{align}
	for some $\lambda_2 \asymp \sqrt{(\log p)/n}$. Mimicking the development for event $\Ev_1$, we now show that $\Ev_2$ also holds with high probability.
	This result is stated in Theorem~\ref{thm:lambda2}.
	Similar to the proof of Theorem~\ref{thm:lambda1}, we first consider the asymptotic property for a particular node $j$ and permutation $\pi$, and then leverage this to get a uniform bound across all permutations and nodes.
	For this proof, we use the following lemma stated in~\cite{cai2015joint}, which also follows from~\cite{zaitsev1987gaussian}. After the lemma, we formally state our result.
	
	\begin{lemma}
		\cite[Lemma 2]{cai2015joint}
		\label{lm:lambda2}
		Suppose $X_1, \cdots, X_n$ are $K$-dimensional random vectors satisfying $\E X_i = 0$ and $\Vert X_i \Vert_2 \leq M$ for $1 \leq i \leq n$. We have for any $\beta > 0$ and $x > \beta$
		\begin{align*}
		\PP(\Vert \textstyle\sum_{i=1}^n X_i \Vert_2 \geq x) \leq & \PP \left( \Vert N \Vert_2 \geq (x - \beta) / \lambda_{\max}^{1/2} \right)\\
		&\qquad + c_1 K^{5/2} \exp(-c_2 K^{-5/2} \beta / M),
		\end{align*}
		where $\lambda_{\max}$ is the largest eigenvalue of $\textrm{Cov}(\sum_{i=1}^n X_i)$, $N$ is a $K$-dimensional standard normal random vector and $c_1, c_2$ are positive constants.
	\end{lemma}
	
	\begin{theorem} \label{thm:lambda2}
		Assume Conditions~\ref{cd6} and~\ref{cd9} hold, then there exist constants $c_1, c_2, c_3>0$ such that
		\begin{align*}
		\PP \Bigg( \exists \pi: \sum_{j=1}^p \sum_{k=1}^K w_k \left( \frac{\Vert \hat{\epsilon}_{j\pi}^{(k)} \Vert_2^2 / n_k - \omega_{0j\pi}^{(k)}}{\omega_{0j\pi}^{(k)}} \right)^2 \geq c_1 \frac{c_sp\log p + \vert \G_{0\pi}^\un\vert \log p}{n} \Bigg)
		\\
		\leq c_2 \exp(-c_3 \log p),
		\end{align*}
		where $c_s$ is the constant defined in Condition~\ref{cd6}.
	\end{theorem}
	
	\begin{proof}
		We begin by analyzing the asymptotic properties of $\epsilon_{j\pi}^{(k)}$ for all $k$ given a fixed permutation $\pi$ and node $j$. More specifically, consider the following random event
		\begin{align} \label{eq:lambda21}
		\C_{j\pi} := \left\lbrace \sum_{k=1}^K w_k \left( \frac{\Vert \hat{\epsilon}_{j\pi}^{(k)} \Vert_2^2 / n_k - \omega_{0j\pi}^{(k)}}{\omega_{0j\pi}^{(k)}} \right)^2 \geq c_1^2 \frac{\log p \,  (|{\pa}_{j}(\G_{0\pi}^{\un})|+c_s)}{n} \right\rbrace,
		\end{align}
		for some positive constant $c_1$. 
		Following the proof of Theorem~1 in~\cite{cai2015joint}, we define $u_t^{(k)}$ as follows:
		\begin{equation*}
		u_t^{(k)} :=
		\begin{cases}
		\frac{\sqrt{w_k}}{n_k} \left( \frac{{[\hat{\epsilon}_{j\pi}^{(k)}]}_t^2 - \omega_{0j\pi}^{(k)}}{\omega_{0j\pi}^{(k)}} \right) \quad &\mathrm{if} \,\, t \leq n_k,\\
		0 \qquad &\mathrm{otherwise}.
		\end{cases}
		\end{equation*}
		Denoting by $u_t = (u_t^{(1)}, \cdots, u_t^{(K)})^T$ the random vector collecting all $u_t^{(k)}$, by definition it follows that
		\begin{align*}
		\sum_{k=1}^K w_k \left( \frac{\Vert \hat{\epsilon}_{j\pi}^{(k)} \Vert_2^2 / n_k - \omega_{0j\pi}^{(k)}}{\omega_{0j\pi}^{(k)}} \right)^2 = \left\Vert \sum_{t=1}^n u_t \right\Vert_2^2.
		\end{align*}
		A straightforward substitution in \eqref{eq:lambda21} allows us to rewrite the event $\C_{j\pi}$ as
		\begin{align}\label{E:def_lam_e_2}
		\C_{j\pi} := \left\lbrace \left\Vert \sum_{t=1}^n u_t \right\Vert_2 \geq c_1 \lambda_n \right\rbrace \quad \textrm{with} \quad \lambda^2_n := \frac{\log p \, (|{\pa}_{j}(\G_{0\pi}^{\un})|+c_s)}{n}.
		\end{align}
		Notice that in order to apply Lemma~\ref{lm:lambda2} to bound the probability of occurrence of $\C_{j\pi}$, we would need $\| u_t \|_2$ to be smaller than a constant $M$, which is not true in general. 
		Hence, we split $\C_{j\pi}$ into two subevents, enabling the utilization of Lemma~\ref{lm:lambda2}.
		More precisely, whenever the following random event holds
		\begin{align*}
		\F_a := \left\lbrace \vert u_t^{(k)}\vert \leq \lambda_n^{-1} K^{1/2 - a}/n, \quad \forall t, k\right\rbrace,
		\end{align*}
		then the $\ell_2$ norm of $u_t$ is bounded as follows: 
		\begin{align*}
		\| u_t \|_2 \leq M_a := \lambda_n^{-1} K^{1 - a} / n,
		\end{align*}
		where $a$ is a free parameter that will be fixed later in the proof.	
		In detail, we bound the probability of $\C_{j\pi}$ according to 
		\begin{equation}\label{E:bayes_theorem}
		\PP(\C_{j\pi}) \leq \PP(\C_{j\pi} \vert \F_a) \PP(\F_a) + \PP(\neg \F_a),
		\end{equation}
		where we use Lemma~\ref{lm:lambda2} to bound $\PP(\C_{j\pi} \vert \F_a)$ and where $\PP(\neg \F_a)$ can be estimated from the chi-squared tail bound~\cite{BEN62}.  
		
		We first focus on bounding $\PP(\C_{j\pi} \vert \F_a)$. 
		For this, we introduce a new variable $\tilde{u}_t$ obtained by truncating the tail of $u_t$. Formally, recalling that $\mathbf{1}\{ \cdot \}$ denotes the indicator function, we have that $\tilde{u}_t := (\tilde{u}_t^{(1)}, \cdots, \tilde{u}_t^{(K)})$ where
		\begin{align*}
		\tilde{u}_t^{(k)} := u_t^{(k)} \mathbf{1} \left\lbrace \vert u_t^{(k)} \vert \leq \lambda_n^{-1} K^{1/2 - a} / n  \right\rbrace - \E\left[ u_t^{(k)} \mathbf{1} \left\lbrace \vert u_t^{(k)} \vert \leq \lambda_n^{-1} K^{1/2 - a} /n \right\rbrace \right].
		\end{align*}
		Notice that whenever the random event $\F_a$ holds, then $u_t$ and $\tilde{u}_t$ follow the same distribution except for a shift $v_t^{(k)} := \E\left[ u_t^{(k)} \mathbf{1} \left\lbrace \vert u_t^{(k)} \vert \leq \lambda_n^{-1} K^{1/2 - a} / n \right\rbrace \right]$. Putting it differently, the distribution of $u_t^{(k)} - \tilde{u}_t^{(k)}$ is a constant $v_t^{(k)}$ whenever we are on the random event $\F_a$.
		This implies that $\left\Vert \sum_{t=1}^n u_t \right\Vert_2 \leq n \;\max_{t, k}\; \vert v_t^{(k)} \vert + \left\Vert \sum_{t=1}^n \tilde{u}_t \right\Vert_2$.
		
		Therefore, if we can guarantee that $n \vert v_t^{(k)}\vert = o(1) \lambda_n$ for all $t$ and $k$, then there must exist a constant $0 < \delta < 1$ such that if we are on the random event $\F_a$, $\left\Vert \sum_{t=1}^n u_t \right\Vert_2 \geq c_1 \lambda_n$ implies that $\left\Vert \sum_{t=1}^n \tilde{u}_t \right\Vert_2 \geq (1 - \delta) c_1 \lambda_n$. Equivalently, we may write
		\begin{align}\label{E:proof_bayes_lambda}
		\PP(\C_{j \pi} \vert \F_a) \PP(\F_a) & \leq \PP\left(\left\Vert \sum_{t=1}^n \tilde{u}_t \right\Vert_2 \geq (1 - \delta) c_1 \lambda_n \Bigg\vert \F_a \right) \PP(\F_a)\\
		& \leq \PP\left(\left\Vert \sum_{t=1}^n \tilde{u}_t \right\Vert_2 \geq (1 - \delta) c_1 \lambda_n\right) \nonumber,
		\end{align}
		where the second inequality follows from Bayes' theorem, and we can bound the last term by applying Lemma~\ref{lm:lambda2} since $\tilde{u}_t$ is bounded by definition. 
		We now show that, indeed, $n \vert v_t^{(k)}\vert = o(1) \lambda_n$ for all $t$ and $k$. From the cumulative tail bound of a chi-squared random variable with one degree of freedom we have that $\PP(u_t^{(k)} \geq l) \leq \exp(- \eta n l/\sqrt{K})$ for some constant $\eta > 0$.
		Based on this, we can estimate the scale of $v_t^{(k)}$ with respect to $p$ and $n$ as
		\begin{align*}
		\begin{split}
		 \vert v_t^{(k)} \vert & \! = \! \left\vert \E\left[ u_t^{(k)} \mathbf{1} \! \left\lbrace \vert u_t^{(k)} \vert \leq \lambda_n^{-1} K^{1/2 - a}/n \right\rbrace \right] \right\vert
		\\
		& \! = \! \left\vert \E\left[ u_t^{(k)} \mathbf{1} \! \left\lbrace \vert u_t^{(k)} \vert > \lambda_n^{-1} K^{1/2 - a}/n \right\rbrace \right] \right\vert \leq \exp \left(- \eta' \lambda_n^{-1} K^{-a} \right)
		\end{split}
		\end{align*}
		for some $0 < \eta' < \eta$, where the second equality follows from the fact that $u_t^{(k)}$ has zero mean.
		Notice that in the last inequality, $u_t^{(k)}$ has been absorbed into the exponential term. 
		As $\vert v_t^{(k)} \vert$ decays exponentially with respect to $\lambda_n^{-1}$, we have that $n \vert v_t^{(k)} \vert = o(1) \lambda_n$. 
		Having justified this, we may now apply Lemma~\ref{lm:lambda2} to the rightmost term in \eqref{E:proof_bayes_lambda} in order to bound $\PP(\C_{j \pi} \vert \F_a) \PP(\F_a)$.
		From the definition of $\tilde{u}_t$ it follows that $\lambda_{\max} \lbrace \textrm{Cov}(\sum_{t=1}^n \tilde{u}_t)\rbrace \leq \lambda_{\max} \lbrace \textrm{Cov}(\sum_{t=1}^n u_t)\rbrace$. Furthermore, since the variables $u_t^{(k)}$ are independently distributed for all $t$, we have that
		\begin{align*}
		\textrm{var}\left(\sum_{t=1}^n u_t^{(k)}\right) = \frac{w_k}{n_k} \textrm{var}\left( \frac{{[\hat{\epsilon}_{j\pi}^{(k)}]}_t^2 - \omega_{0j\pi}^{(k)}}{\omega_{0j\pi}^{(k)}} \right) \leq c_2 /n
		\end{align*}
		for some constant $c_2 > 0$. This also implies that $\lambda_{\max} \lbrace \textrm{Cov}(\sum_{t=1}^n \tilde{u}_t)\rbrace \leq c_2 / n$. 
		Applying Lemma~\ref{lm:lambda2}, where we select $x = (1 - \delta) c_1 \lambda_n$, $\beta = \delta' x$ for some arbitrary positive constant $0 < \delta' < 1$ and, $M = \lambda_n^{-1} K^{1 - a} / n$, it follows that
		\begin{align*}
		\PP\left(\left\Vert \sum_{t=1}^n \tilde{u}_t \right\Vert_2 \geq (1 - \delta) c_1 \lambda_n \right) &
		\leq \exp(-c_3 (n \lambda^2_n - K)) \\
		&\quad + \exp\left( \frac{2}{5} \log K - c_4 K^{a - \frac{7}{2}} n \lambda^2_n \right),
		\end{align*}
		for some constants $c_3, c_4 > 0$ that increase if constant $c_1$ is increased.
		In addition, by choosing $a=7/2$, there must exist a large enough $c_1$ such that $c_3, c_4 > 1$ and therefore
		\begin{align}\label{E:proof_first_term}
		& \PP\left(\left\Vert \sum_{t=1}^n \tilde{u}_t \right\Vert_2 \geq (1 - \delta) c_1 \lambda_n\right)\\
		& \quad \leq  \exp(-n\lambda_n^2) = \exp(-|{\pa}_{j}(\G^{\un})| \log p - c_s \log p).\nonumber
		\end{align}
		Replacing \eqref{E:proof_first_term} into \eqref{E:proof_bayes_lambda} gives us the sought exponential bound for the first summand in~\eqref{E:bayes_theorem}. 
		
		We are now left with the task of finding a bound for $\PP(\neg \F_a)$.
		By relying on the fact that $n^{(1)} \asymp \cdots \asymp n^{(K)}$ (cf. Condition~\ref{cd9}), we get that $(\max_k \; n_k) \, K \leq c_5 n$ for some constant $c_5$ and therefore
		\begin{align*}
		\PP(\neg \F_a) \leq c_5 n \underset{1 \leq k \leq K, 1 \leq l \leq n}{\max} \PP\left\lbrace \vert u_t^{(k)} \vert \geq \lambda_n^{-1} K^{1/2 - a} /n \right\rbrace.
		\end{align*}
		Following this, in order to bound the probability that $\neg \F_a$ holds we further rely on the tail bound of the chi-squared random variable $u_t^{(k)}$ to obtain
		\begin{align*}
		\PP(\neg \F_a)\leq c_5 n \exp \left( - \eta \lambda_n^{-1} K^{-a} \right) = c_5 \exp \left( \log n - \eta \lambda_n^{-1} K^{-a} \right).
		\end{align*}
		Recalling the definition of $\lambda_n$ from \eqref{E:def_lam_e_2}, Condition~\ref{cd6} implies that
		\begin{align}\label{E:two_eq_e_2}
		\frac{\lambda_n^{-1}}{K^{7 / 2}} \geq \log p (|{\pa}_{j}(\G^{\un})| + c_s) / \tilde{\alpha}^{\frac{3}{2}} \quad\textrm{and}\quad \sqrt{\tilde{\alpha}}\frac{\lambda_n^{-1}}{K^{7 / 2}} \geq \log n.
		\end{align}
		By recalling that we have fixed $a = \frac{7}{2}$, 	it follows that there exists a constant $\eta'$ such that
		\begin{align*}
		\PP(\neg \F_a) & \leq c_5 \exp \left(\log n - \eta \lambda_n^{-1} K^{-a} \right) \\
		& \leq c_5 \exp \left(\sqrt{\tilde{\alpha}} \lambda_n^{-1} K^{-a} - \eta \lambda_n^{-1} K^{-a} \right) \leq c_5 \exp \left( - \eta' \lambda_n^{-1} K^{-a} \right),
		\end{align*}
		where we have used the second inequality in~\eqref{E:two_eq_e_2}. Furthermore, by leveraging the first inequality in~\eqref{E:two_eq_e_2} we obtain that
		\begin{align*}
		\PP(\neg \F_a) \leq c_5 \exp \left( - \eta' \lambda_n^{-1} K^{-a} \right) \leq c_5 \exp \left( - \eta' \log p \, (|{\pa}_{j}(\G^{\un})| + c_s) / \tilde{\alpha}^{\frac{3}{2}} \right),
		\end{align*}
		thus obtaining an exponential bound for the second summand in~\eqref{E:bayes_theorem}.
		
		Having found exponential bounds for both summands in~\eqref{E:bayes_theorem}, it follows that for $c_1 > 0$ sufficiently large and $\tilde{\alpha}$ sufficiently small 
		we have that 
\begin{align*}
\PP(\C_{j \pi}) \leq (1 + c_5) \exp(-|{\pa}_{j}(\G^{\un})| \log p - c_s \log p).
\end{align*}	
		Following an argument based on union bounds similar to the one presented in the proof of Theorem~~\ref{thm:lambda1}, we have that
		\begin{align*}
		&\PP \left(\sum_{k=1}^K w_k \left( \frac{\Vert \hat{\epsilon}_{j\pi}^{(k)} \Vert_2^2 / n_k - \omega_{0j\pi}^{(k)}}{\omega_{0j\pi}^{(k)}} \right)^2 \leq c_1 \frac{\log p \, (|{\pa}_{j}(\G^{\un})|+c_s)}{n}, \quad \forall\; j, \pi \right) \\
		& \hspace{1cm} \geq 1 - \PP \left( \bigcup_{j=1}^p \bigcup_{m=1}^p \bigcup_{\pi \in \Pi_j(m)} \C_{j\pi} \right) \geq 1 - (1 + c_5) \exp (-(c_s - 2)\log p)
		\end{align*}
		for some constant $c_1 > 0$. It is immediately implied from the previous expression that
		\begin{align*}
		\PP \! \left(\! \exists \pi :  \sum_{j=1}^p \sum_{k=1}^K w_k \left( \frac{\Vert \hat{\epsilon}_{j\pi}^{(k)} \Vert_2^2 / n_k - {\omega_0}_{j\pi}^{(k)}}{{\omega_0}_{j\pi}^{(k)}} \right)^2 \!\! \! \geq \! c_1 \frac{c_s p \log p + \vert \G_{0\pi}^\un\vert \log p}{n} \right) \!\! \\
		\leq \! (1 + c_5) \exp (- (c_s - 2) \log p),
		\end{align*}
		thus recovering the statement of the theorem (since $c_s>2$ by Assumption~\ref{cd6}) after accordingly renaming the constants on the right-hand side.
	\end{proof}
	
	\subsubsection{Random event $\Ev_3$}\label{Sss:random_event_3}
	
	Event $\Ev_3$ is defined as the intersection of $2K$ events that we denote by
	$\{\Ev^{(k)}_{3a}\}_{k=1}^K$ and $\{\Ev^{(k)}_{3b}\}_{k=1}^K$, where events
	$\Ev^{(k)}_{3a}$ and $\Ev^{(k)}_{3b}$ are specific to the $k$th SEM.
	More specifically, event $\Ev^{(k)}_{3a}$ ensures that all the estimated noise variances $\hat{\omega}_j^{(k)}$ associated with the $k$th SEM are finite and bounded away from zero. Formally, we define the following events for $k=1, \ldots, K$:
	\begin{equation}\label{E:event_3a}
	\Ev^{(k)}_{3a} := \left\{ \min\left(\hat{\omega}_j^{(k)}, 1/\hat{\omega}_j^{(k)}\right) \geq 1/\beta^2, \quad \mathrm{for} \,\, j = 1, \ldots, p \right\},
	\end{equation}
	for some $\beta > 0$.
	Event $\Ev^{(k)}_{3b}$ imposes a universal lower bound on the norm achievable by any linear combinations of the data associated with the $k$-th DAG.
	Mathematically, we consider the ensuing events for $k = 1, \ldots, K$: 
	\begin{equation}\label{E:event_3b}
	\Ev^{(k)}_{3b} := \left\{  \Vert \hat{X}^{(k)} v \Vert_2 / \sqrt{n_k} \geq \left(\delta_3 - \lambda_3^{(k)} \sqrt{\Vert v \Vert_0} \right) \Vert v \Vert_2, \quad \forall v \in \R^p \right\},
	\end{equation}
	for some $\delta_3 > 0$ and $\lambda_3^{(k)} \asymp \sqrt{(\log p)/n_k}$.
	Based on \eqref{E:event_3a} and \eqref{E:event_3b} we define events $\Ev^{(k)}_{3} := \Ev^{(k)}_{3a} \cap \Ev^{(k)}_{3b}$, and
	\begin{equation}\label{E:event_3}
	\Ev_3 := \bigcap_{k=1}^K \Ev^{(k)}_{3}.
	\end{equation}
	Leveraging the fact that Condition~\ref{cd5} enforces the maximum in-degree of each $\G_{0 \pi}^{(k)}$ to be at most $\alpha n_k / \log p$ for some positive constant $\alpha$, we can generalize Lemmas~7.5 and~7.7 from~\cite{vandegeer2013} into the following lemma.
	\begin{lemma}
		\cite[Lemmas~7.5 and~7.7]{vandegeer2013}
		\label{lm:lambda03}
		Assume Conditions~\ref{cd3},~\ref{cd4},~\ref{cd5} and~\ref{cd6} hold and that
		\begin{align*}
		3 \sqrt{\Lambda_{\min}} / 4 - \sqrt{\frac{2(t+ \log p)}{n}} - 3 \sigma_0 \sqrt{\alpha + \tilde{\alpha}} \geq 1 / \beta > 0,
		\end{align*}
		for some $t > 0$. Based on this, define
		\begin{align*}
		{\lambda_3^{(k)}} := 3 \sigma_0 \sqrt{\frac{\log p}{n_k}}, \quad \mathrm{and} \quad \delta_3 := 3 \sqrt{\Lambda_{\min}} / 4 - \sqrt{\frac{2(t+\log p)}{n_k}}.
		\end{align*}
		Then $\PP( \Ev^{(k)}_{3} ) \geq 1 - 5 \exp(-t)$
		and on $\Ev^{(k)}_{3}$ it holds that
		\begin{align}\label{eq:frobenius}
		\Vert \hat{X}^{(k)} (a_j^{(k)} - a_{0j\hat{\pi}}^{(k)})\Vert_2 / \sqrt{n_k} \geq \Vert a_j^{(k)} - a_{0j\hat{\pi}}^{(k)} \Vert_2 / \beta^2.
		\end{align}
	\end{lemma}
	
	Lemma~\ref{lm:lambda03} shows that under certain conditions the events $\Ev^{(k)}_{3}$ hold with high probability, thus playing a role analogous to that of Theorem~\ref{thm:lambda1} for event $\Ev_{1}$ and Theorem~\ref{thm:lambda2} for event $\Ev_{2}$.
	
	With the events $\Ev_{1}$, $\Ev_{2}$, and $\Ev_{3}$ defined and having shown under which conditions these hold with high probability, in the next section we leverage these events to prove Theorem~\ref{thm:l0}.
	
	\subsection{Proof of Theorem~\ref{thm:l0}} \label{sec:l0proof}
	
	\subsubsection{Bounds on new probability space}
	
	Through direct manipulation of the likelihood function, in Lemma~\ref{lem:bounds} we show that the global optimum $(\hat{A}^{(k)}, \hat{\Omega}^{(k)})_{k=1}^K$ converges to the SEMs $(A_{0 \hat{\pi}}^{(k)}, \Omega_{0 \hat{\pi}}^{(k)})_{k=1}^K$, where $\hat{\pi}$ is some permutation consistent with the estimated adjacency matrices $\hat{A}^{(1)}, \cdots, \hat{A}^{(K)}$.

	\begin{lemma} \label{lem:bounds}
		Assume we are on $\Ev_1 \cap \Ev_2 \cap \Ev_3$ and Condition~\ref{cd3} holds. Consider a regularizer in \eqref{obj:jointl0} satisfying $\lambda^2 > \lambda_1^2/\delta_1 + \lambda_2^2 / \delta_2$ with $0 < \delta_1 < 1 / \beta^2$ and $0 < \delta_2 < 1 / (2 \beta^2 \sigma_0^2)$. Then, $\left\{\hat{\pi}, \{(\hat{A}^{(k)}, \hat{\Omega}^{(k)})\}_{k=1}^K \right\}$ the global optimum of \eqref{obj:jointl0}, satisfies
		\begin{align}\label{eq:bound4}
		& \left( \frac{1}{\beta^2} - \delta_1\right) \sum_{j=1}^p \sum_{k=1}^K w_k \Vert X^{(k)} (\hat{a}_j^{(k)} - a_{0j\hat{\pi}}^{(k)}) \Vert_2^2 / n_k \\
		&\quad + \left( \frac{1}{2\beta^4\sigma_0^4} - \delta_2\right) \sum_{j=1}^p \sum_{k=1}^K w_k \left( \frac{\hat{\omega}_j^{(k)} - {\omega_0}_j^{(k)}}{\hat{\omega}_j^{(k)}} \right)^2 \nonumber \\
		&\quad + \left( \lambda^2 - \frac{\lambda_1^2}{\delta_1} - \frac{\lambda_2^2}{\delta_2} \right) \vert \hat{\G} \vert \leq \lambda^2 \vert \G_0^\un \vert + \frac{\lambda_2^2 (p + \vert \G_{0\hat{\pi}}^\un \vert)}{\delta_2} + \frac{\lambda_1^2 \vert \G_{0\hat{\pi}}^\un \vert}{\delta_1}. \nonumber
		\end{align}
	\end{lemma}
	
	\begin{proof}
		By definition, the global optimum must satisfy
		\begin{align} \label{eq:globalobj}
		\sum_{k=1}^K w_k \ell_{n_k}(\hat{X}^{(k)}; \hat{A}^{(k)}, \hat{\Omega}^{(k)}) - \lambda^2 \vert \hat{\G} \vert \geq \sum_{k=1}^K w_k \ell_{n_k}(\hat{X}^{(k)}; A_0^{(k)}, \Omega_0^{(k)}) - \lambda^2 \vert \G_0^\un \vert.
		\end{align}
		Let $\hat{\pi}$ denote any permutation consistent with all $\hat{A}^{(k)}$. 
		Since the value of the likelihood $\ell_{n_k}(\hat{X}^{(k)}; A_0^{(k)}, \Omega_0^{(k)})$ is completely determined by the precision matrices $\{\Theta_0^{(k)}\}_{k=1}^K$, it then follows that the likelihood function $\ell_{n_k}(\hat{X}^{(k)}; A_0^{(k)}, \Omega_0^{(k)})$ and the function $\ell_{n_k}(\hat{X}^{(k)}; A_{0\hat{\pi}}^{(k)}, \Omega_{0\hat{\pi}}^{(k)})$ achieve the same value. 
		We therefore replace the former by the latter in \eqref{eq:globalobj} and expand the definition of the likelihood function in \eqref{eq:likelihood} to obtain
		\begin{align*}
		& p + \sum_{j=1}^p \sum_{k=1}^K w_k \log \hat{\omega}_j^{(k)} + \lambda^2 \vert \hat{\G} \vert \\
		&\quad \leq \sum_{j=1}^p \sum_{k=1}^K w_k \frac{\Vert \hat{\epsilon}_{j\hat{\pi}}^{(k)} \Vert_2^2 / n_k}{\omega_{0j\hat{\pi}}^{(k)}} + \sum_{j=1}^p \sum_{k=1}^K w_k \log \omega_{0j\hat{\pi}}^{(k)} + \lambda^2 \vert \G_0^\un \vert.
		\end{align*}
		Basic manipulations transform the above expression into the following inequality 
		\begin{align}\label{eq:bounds1}
		\sum_{j=1}^p \sum_{k=1}^K w_k \log \left( \frac{\hat{\omega}_j^{(k)}}{\omega_{0j\hat{\pi}}^{(k)}} \right) + \lambda^2 \vert \hat{\G} \vert \leq \sum_{j=1}^p \sum_{k=1}^K w_k \left( \frac{\Vert \hat{\epsilon}_{j\hat{\pi}}^{(k)} \Vert_2^2 / n_k}{\omega_{0j\hat{\pi}}^{(k)}} - 1 \right) + \lambda^2 \vert \G_0^\un \vert.
		\end{align}
		Since we are on $\Ev_3$, we have that $1 / \hat{\omega}_j^{(k)} \leq \beta^2$ [cf.~\eqref{E:event_3a}]. 
		By combining this with Condition~\ref{cd3} we can further bound $\omega_{0j\hat{\pi}}^{(k)} / \hat{\omega}_j^{(k)} \leq \beta^2 \sigma_0^2$. Then, using the Taylor expansion $\log (1+x) \leq x - x^2 / (2(1+t)^2)$, for $-1 < x \leq t$, we can further replace $\log \left( \frac{\hat{\omega}_j^{(k)}}{\omega_{0j\hat{\pi}}^{(k)}} \right)$ in \eqref{eq:bounds1} to obtain
		\begin{align}\label{eq:bounds1_11}
		\sum_{j=1}^p \sum_{k=1}^K w_k & \left( \frac{\hat{\omega}_j^{(k)} - \omega_{0j\hat{\pi}}^{(k)}}{\hat{\omega}_j^{(k)}}\right) + \frac{1}{2\beta^4 \sigma_0^4} \left( \frac{\omega_{0j\hat{\pi}}^{(k)}}{\hat{\omega}_j^{(k)}} - 1 \right) ^2 + \lambda^2 \vert \hat{\G} \vert \nonumber\\
		& \leq \sum_{j=1}^p \sum_{k=1}^K w_k \left( \frac{\Vert \hat{\epsilon}_{j\hat{\pi}}^{(k)} \Vert_2^2 / n_k}{\omega_{0j\hat{\pi}}^{(k)}} - 1 \right) + \lambda^2 \vert \G_0^\un \vert.
		\end{align}
		Finally, using the fact that $\hat{X}_j^{(k)} = \hat{\epsilon}_{j\hat{\pi}}^{(k)} + \hat{X}^{(k)} a_{0j\hat{\pi}}^{(k)}$, we also rewrite $\hat{\omega}_j^{(k)}$ as
		\begin{align*}
		& \hat{\omega}_j^{(k)} = \Vert \hat{X}_j^{(k)} - \hat{X}^{(k)} \hat{a}_j^{(k)} \Vert_2^2 / n_k\\
		& = \Vert \hat{X}^{(k)} (\hat{a}_j^{(k)} - a_{0j\hat{\pi}}^{(k)}) \Vert_2^2 / n_k - 2 \hat{\epsilon}_{j\hat{\pi}}^{(k)T} \hat{X}^{(k)} (\hat{a}_j^{(k)} - a_{0j\hat{\pi}}^{(k)}) / n_k+ \Vert \hat{\epsilon}_j^{(k)} \Vert_2^2 / n_k.
		\end{align*}
		By replacing the above into \eqref{eq:bounds1_11}, we get that
		\begin{align}\label{eq:eqt3}
		& \sum_{j=1}^p \sum_{k=1}^K w_k \frac{\Vert \hat{X}^{(k)} (\hat{a}_j^{(k)} - a_{0j\hat{\pi}}^{(k)}) \Vert_2^2 / n_k}{\hat{\omega}_j^{(k)}} + \frac{1}{2\beta^4\sigma_0^4}\sum_{j=1}^p \sum_{k=1}^K w_k \left( \frac{\hat{\omega}_j^{(k)} - \omega_{0j\hat{\pi}}^{(k)}}{\hat{\omega}_j^{(k)}} \right)^2 \!\!\!+\! \lambda^2 \vert \hat{\G} \vert \nonumber\\
		& \hspace{0.5cm}\leq 2\sum_{j=1}^p \sum_{k=1}^K w_k \frac{\hat{\epsilon}_{j\hat{\pi}}^{(k)T} \hat{X}^{(k)} (\hat{a}_j^{(k)} -  a_{0j\hat{\pi}}^{(k)}) / n_k}{\hat{\omega}_j^{(k)}} + \sum_{j=1}^p \sum_{k=1}^K w_k \left( \frac{\Vert \hat{\epsilon}_{j\hat{\pi}}^{(k)} \Vert_2^2 / n_k}{\omega_{0j\hat{\pi}}^{(k)}} - 1 \right) \nonumber \\
		& \hspace{1cm}- \sum_{j=1}^p \sum_{k=1}^K w_k \left( \frac{\Vert \hat{\epsilon}_{j\hat{\pi}}^{(k)} \Vert_2^2 / n_k - \omega_{0j\hat{\pi}}^{(k)}}{\hat\omega_j^{(k)}} \right) + \lambda^2 \vert \G_0^\un \vert.
		\end{align}
		In order to further bound the expression in \eqref{eq:eqt3}, notice that the first summand in the right hand side of the inequality corresponds to the sum of all empirical correlation coefficients.
		Leveraging that we are under the assumption that $\Ev_1$ holds [cf.~\eqref{E:event_E_1}], we have that
		\begin{align}\label{E:bound_first_summand}
		2\sum_{j=1}^p \sum_{k=1}^K w_k & \frac{\hat{\epsilon}_{j\hat{\pi}}^{(k)T} \hat{X}^{(k)} (\hat{a}_j^{(k)} -  a_{0j\hat{\pi}}^{(k)}) / n_k}{\hat{\omega}_j^{(k)}} \nonumber\\
		& \leq \delta_1\sum_{j=1}^p \sum_{k=1}^K w_k \Vert \hat{X}^{(k)} (\hat{a}_j^{(k)} - a_{0j\hat{\pi}}^{(k)}) \Vert_2^2 / n_k + \frac{\lambda_1^2}{\delta_1} \vert \hat{\G} \vert.
		\end{align}
		In order to bound the second and third terms, we first restate their difference as follows
		\begin{align}\label{eq:eqt2}
		\sum_{j=1}^p \sum_{k=1}^K w_k & \! \left( \! \frac{\Vert \hat{\epsilon}_{j\hat{\pi}}^{(k)} \Vert_2^2 / n_k - \omega_{0j\hat{\pi}}^{(k)}}{\hat\omega_j^{(k)}} \! \right) \!- \!w_k \left( \! \frac{\Vert \hat{\epsilon}_{j\hat{\pi}}^{(k)} \Vert_2^2 / n_k}{\omega_{0j\hat{\pi}}^{(k)}} - 1 \! \right) \nonumber \\
		& = \! \sum_{j=1}^p \sum_{k=1}^K w_k \! \left(\! \frac{\Vert \hat{\epsilon}_{j\hat{\pi}}^{(k)} \Vert_2^2 / n_k - \omega_{0j\hat{\pi}}^{(k)}}{\omega_{0j\hat{\pi}}^{(k)}} \! \right) \!\!\! \left(\! \frac{\omega_{0j\hat{\pi}}^{(k)} - \hat\omega_j^{(k)}}{\hat\omega_j^{(k)}} \! \right)\!\!.
		\end{align}
Next, by using Cauchy-Schwarz inequality, i.e. 
\begin{align*}
\vert \sum_{i=1}^n u_i v_i \vert^2 \leq \sum_{j=1}^n \vert u_j\vert^2 \sum_{k=1}^n \vert v_k\vert^2,
\end{align*}		
we further bound~\eqref{eq:eqt2} as
		\begin{align}\label{eq:eqt2_22}
		\begin{split}
		& \left\vert \sum_{j=1}^p \sum_{k=1}^K w_k \left( \frac{\Vert \hat{\epsilon}_{j\hat{\pi}}^{(k)} \Vert_2^2 / n_k - \omega_{0j\hat{\pi}}^{(k)}}{\hat\omega_j^{(k)}} \right) - \sum_{j=1}^p \sum_{k=1}^K w_k \left( \frac{\Vert \hat{\epsilon}_{j\hat{\pi}}^{(k)} \Vert_2^2 / n_k}{\omega_{0j\hat{\pi}}^{(k)}} - 1 \right) \right\vert \\
		& \, \leq \!\! \left( \sum_{j=1}^p \sum_{k=1}^K w_k \left( \frac{\Vert \hat{\epsilon}_{j\hat{\pi}}^{(k)} \Vert_2^2 / n_k - \omega_{0j\hat{\pi}}^{(k)}}{\omega_{0j\hat{\pi}}^{(k)}} \right)^2 \right)^{1 / 2} \left( \sum_{j=1}^p \sum_{k=1}^K w_k \left( \frac{ \omega_{0j\hat{\pi}}^{(k)} - \hat\omega_j^{(k)}}{\hat\omega_j^{(k)}} \right)^2 \right)^{1 / 2}.
		\end{split}
		\end{align}
		From the fact that event $\Ev_2$ holds [cf.~\eqref{E:event_2}], we can upper bound the first of the two factors in the right-hand side of \eqref{eq:eqt2_22} by $2\sqrt{ \lambda_2^2 (p + \vert \G_0^\un (\hat{\pi})\vert)}$. 
		Further, relying on the inequality $2ab \leq a^2 /\delta_2 + \delta_2 b^2$ for any $\delta_2 > 0$, it follows that
		\begin{align}\label{eq:eqt2_33}
		&\left\vert \sum_{j=1}^p \sum_{k=1}^K w_k \left( \frac{\Vert \hat{\epsilon}_{j\hat{\pi}}^{(k)} \Vert_2^2 / n_k - \omega_{0j\hat{\pi}}^{(k)}}{\hat\omega_j^{(k)}} \right) - \sum_{j=1}^p \sum_{k=1}^K w_k \left( \frac{\Vert \hat{\epsilon}_{j\hat{\pi}}^{(k)} \Vert_2^2 / n_k}{\omega_{0j\hat{\pi}}^{(k)}} - 1 \right) \right\vert \\
		& \qquad\qquad \qquad \qquad \leq \frac{\lambda_2^2 (p + \vert \G_0^\un (\hat{\pi})\vert)}{\delta_2} + \delta_2 \sum_{j=1}^p \sum_{k=1}^K w_k \left( \frac{\omega_{0j\hat{\pi}}^{(k)} - \hat\omega_j^{(k)}}{\hat{\omega}_j^{(k)}} \right)^2. \nonumber
		\end{align}
		By replacing~\eqref{E:bound_first_summand} and \eqref{eq:eqt2_33} into \eqref{eq:eqt3}, we recover \eqref{eq:bound4}, as we wanted to show.
	\end{proof}
	
	From Lemma~\ref{lem:bounds} it follows that the global optimum of \eqref{obj:jointl0} corresponds to a minimal I-MAP, but no claim is made about the sparsity level of this I-MAP. In order to show that the solution is indeed sparse, we must rely on Conditions~\ref{cd7} and~\ref{cd8}.
	In Lemma~\ref{lm:betamin} we show that $\vert \G_{0\hat{\pi}}^\un \vert$ cannot be much larger than $\vert \hat{\G} \vert$. Then, in Thm.~\ref{thm:betamin} we further show how to cancel out $\vert \G_{0\hat{\pi}}^\un \vert$ with $\vert \hat{\G} \vert$ in \eqref{eq:bound4} to obtain our main result. 
	
	\begin{lemma} \label{lm:betamin}
		Assume Condition~\ref{cd7} holds and let $\tilde{\lambda}>0$ be such that 
		\begin{align} \label{eq:betamin}
		\sum_{k=1}^K w_k \Vert \hat{A}^{(k)} - A_{0\hat{\pi}}^{(k)} \Vert_F^2 \leq \tilde{\lambda}^2 \vert \G_{0\hat{\pi}}^\un \vert.
		\end{align}
		Consider constants $\eta_1, \eta_2$ with $0 \leq \eta_1 < 1$ and $0 < \eta_2^2 c_t < 1 - \eta_1$ such that $\sum_{i,j} \mathbf{1} \left\{ \left| [A_{0\hat{\pi}}^{(k)}]_{i,j} \right| \geq \tilde{\lambda} / \eta_2 \right\} \geq (1 - \eta_1) \vert \G_{0\hat{\pi}}^{(k)}\vert$. Then, it follows that
		\begin{align*}
		\vert \hat{\G} \vert \geq \frac{1 - \eta_1 - \eta_2^2 c_t}{c_t} \vert \G_{0\hat{\pi}}^\un\vert.
		\end{align*}
	\end{lemma}
	
	\begin{proof}
		Let $\mathcal{N}^{(k)}$ and $\mathcal{M}^{(k)}$ be the sets of entries satisfying
\begin{align*}
\mathcal{N}^{(k)} := \lbrace (i,j) :\, \vert [A_{0\hat{\pi}}^{(k)}]_{i,j} \vert \geq \tilde{\lambda} / \eta_2 \rbrace
\end{align*}		
and 
\begin{align*}
\quad\mathcal{M}^{(k)} := \lbrace (i,j) :\, \vert [\hat{A}^{(k)}]_{i,j} - [A_{0\hat{\pi}}^{(k)}]_{i,j} \vert \geq \tilde{\lambda} / \eta_2 \rbrace.
\end{align*}
From these definitions it follows that
		\begin{align*}
		\sum_{k=1}^K w_k \vert \mathcal{N}^{(k)} \intersection \mathcal{M}^{(k)} \vert \frac{\tilde{\lambda}^2}{\eta_2^2} & \leq
		\sum_{k=1}^K w_k \!\!\! \underset{(i,j) \, \in \, \mathcal{N}^{(k)} \cap \mathcal{M}^{(k)}}{\sum} \vert [\hat{A}^{(k)}]_{i,j} - [A_{0\hat{\pi}}^{(k)}]_{i,j} \vert^2 \\
		&\leq 
		\sum_{k=1}^K w_k \Vert \hat{A}^{(k)} - A_{0\hat{\pi}}^{(k)}\Vert_F^2.
		\end{align*}
		Leveraging inequality~\eqref{eq:betamin} and Condition~\ref{cd7} we further have that
		\begin{align}\label{E:proof_lemma_a7}
		\sum_{k=1}^K w_k \vert \mathcal{N}^{(k)} \cap \mathcal{M}^{(k)} \vert \leq \eta_2^2 \vert \G_{0\hat{\pi}}^\un \vert \leq \eta_2^2 c_t \sum_{k=1}^K w_k \vert \G_{0\hat{\pi}}^{(k)} \vert.
		\end{align}
		Notice that for all $(i,j)$-th entries in the set $\mathcal{N}^{(k)} \cap {\mathcal{M}^{(k)}}^\mathcal{C}$ it must be that $| [\hat{A}^{(k)}]_{i,j} | > 0$.
		Hence, $\mathcal{N}^{(k)} \cap {\mathcal{M}^{(k)}}^\mathcal{C}$ corresponds to a subset of non-zero entries of $\hat{A}^{(k)}$, which in turn corresponds to a subset of edges in $\hat{\G}$.
		From this we can infer that
		\begin{align*}
		\vert \hat{\G} \vert \! & = \! \sum_{k=1}^K \! w_k \vert \hat{\G} \vert \! \geq \!
		\sum_{k=1}^K \! w_k \vert \mathcal{N}^{(k)} \!\cap\! {\mathcal{M}^{(k)}}^\mathcal{C} \vert \!=\! \sum_{k=1}^K w_k (\vert \mathcal{N}^{(k)} \vert - \vert \mathcal{N}^{(k)} \cap \mathcal{M}^{(k)} \vert) \\
		&\quad \geq (1 - \eta_1 - \eta_2^2 c_t) \sum_{k=1}^K w_k \vert \G_{0\hat{\pi}}^{(k)} \vert,
		\end{align*}
		where the last inequality follows by combining \eqref{E:proof_lemma_a7} with the definition of $\eta_1$ in the statement of the lemma. The proof concludes by replacing Condition~\ref{cd7} in the above inequality.
	\end{proof}
	
	\begin{theorem} \label{thm:betamin}
		Assume Conditions~\ref{cd1},~\ref{cd7} and~\ref{cd8} hold, and suppose that there exist constants $\delta_B$ and $0 < \delta_s < 1$ as well as $\lambda$ and $\lambda_0$ that scale as $\lambda^2 \asymp \lambda_0^2 \asymp \frac{\log p}{n} (p / \vert \G_0^\un\vert \vee 1)$ such that
		\begin{align}\label{eq:betamin1}
		\delta_B \sum_{k=1}^K w_k \Vert \hat{A}^{(k)} - A_{0\hat{\pi}}^{(k)} \Vert_F^2 + \lambda^2 \delta_s \vert \hat{\G} \vert \leq \lambda^2 \vert \G_0^\un \vert + \lambda_0^2  \vert \G_{0\hat{\pi}}^\un \vert.
		\end{align}
		If the constant $\eta_0$ in Condition~\ref{cd8} is sufficiently small, then there exist constants $\delta_s', c_g, c_g' > 0$ such that
		\begin{align}\label{E:statement_theo_810}
		\delta_B \sum_{k=1}^K w_k \Vert \hat{A}^{(k)} - A_{0\hat{\pi}}^{(k)}\Vert_F^2 + (\lambda^2 \delta_s - \lambda_0^2 \delta_s') \vert \hat{\G} \vert \leq \lambda^2 \vert \G_0^\un \vert
		\end{align}
		and
		\begin{align}\label{E:statement_theo_820}
		\vert \hat{\G} \vert \geq c_g \vert \G_{0\hat{\pi}}^\un \vert \geq c_g' \vert \G_0^\un \vert.
		\end{align}
	\end{theorem}
	
	\begin{proof}
	Using Conditions~\ref{cd1} and~\ref{cd7}, we have that $| \G_0^\un | \leq c_t \, \underset{k}{\max}\; | \G_0^{(k)} | \leq c_t \, \vert \G_{0\hat{\pi}}^\un \vert$.  
		Suppose that for some $\tilde{\lambda} > 0$ one has that $\tilde{\lambda}^2 \asymp \lambda^2 \asymp \lambda_0^2$ and $\tilde{\lambda}^2 \delta_B \geq \lambda^2 c_t + \lambda_0^2$, then it follows from~\eqref{eq:betamin1} that
		\begin{align*}
		\delta_B \sum_{k=1}^K w_k \Vert \hat{A}^{(k)} - A_{0\hat{\pi}}^{(k)} \Vert_F^2 \leq \lambda^2 \vert \G_0^\un \vert + \lambda_0^2 \vert \G_{0\hat{\pi}}^\un \vert \leq \tilde{\lambda}^2 \delta_B \vert \G_{0\hat{\pi}}^\un \vert.
		\end{align*}
		Let $\eta_2$ be a constant defined as $\eta_2 := \eta_0 \tilde{\lambda} / \sqrt{\frac{\log p}{n} (p / \vert \G_0^\un\vert \vee 1)}$, then we can rewrite Condition~\ref{cd8} as
		$$\sum_{i,j} \mathbf{1} \left\{ \left| [A_{0\hat{\pi}}^{(k)}]_{i,j} \right| \geq \tilde{\lambda} / \eta_2 \right\} \geq (1 - \eta_1) \vert \G_{0\hat{\pi}}^{(k)}\vert.$$
		Moreover, for $\eta_0$ sufficiently small, $\eta_2$ is also guaranteed to satisfy $0 < \eta_2^2 c_t < 1 - \eta_1$.
		We could therefore apply Lemma~\ref{lm:betamin} and get that $\vert \hat{\G} \vert \geq \frac{1 - \eta_1 - \eta_2^2 c_t}{c_t} \vert \G_{0\hat{\pi}}^\un \vert$, which completes the proof of~\eqref{E:statement_theo_820} by choosing $c_g = \frac{1 - \eta_1 - \eta_2^2 c_t}{c_t}$ and $c_g' = c_g \cdot c_t$.
		Notice that in order to apply Lemma~\ref{lm:betamin}, it is required that $\eta_2^2 c_t < 1 - \eta_1$, which is guaranteed by the assumption that $\eta_0$ is sufficiently small. 
		Leveraging the first inequality in \eqref{E:statement_theo_820}, we can replace $\lambda_0^2 \vert \G_{0\hat{\pi}}^\un \vert$ in~\eqref{eq:betamin1} by $\frac{\lambda_0^2 c_t}{1 - \eta_1 - \eta_2^2 c_t} \vert \hat{\G} \vert$ in order to obtain
		\begin{align}\label{E:proof_theo_810}
		\delta_B \sum_{k=1}^K w_k \Vert \hat{A}^{(k)} - A_{0\hat{\pi}}^{(k)}\Vert_F^2 + \left(
		\lambda^2 \delta_s - \frac{\lambda_0^2 c_t}{1 - \eta_1 - \eta_2^2 c_t} \right) \vert \hat{\G} \vert\leq \lambda^2 \vert \G_0^\un \vert.
		\end{align}
		Notice that \eqref{E:proof_theo_810} coincides with the sought expression \eqref{E:statement_theo_810} upon substituting $\delta'_s = c_t /(1-\eta_1 - c_t\eta_2^2)$.
	\end{proof}
	
	\subsubsection{Proof of Theorem~\ref{thm:l0}} \label{sec:proof_mainthm}
	
	It follows from Theorem~\ref{thm:lambda1}, Theorem~\ref{thm:lambda2}, and Lemma~\ref{lm:lambda03} that there exist constants $\lambda_1, \lambda_2$ with $\lambda_1^2 \asymp \lambda_2^2 \asymp \frac{\log p}{n}$ as well as some ${\lambda_3^{(k)}}^2 \asymp \frac{\log p}{n_k}$ for all $k$ such that with probability $1 - \exp(-c \log p)$ for some constant $c > 0$, the random event $\Ev_1 \cap \Ev_2 \cap \Ev_3$ occurs.
	
	We may then apply Lemma~\ref{lem:bounds} to show that there exist constants $\delta_B, \delta_W$ such that with high probability, for any $\lambda > 0$ satisfying $\lambda^2 > {\lambda_1^2}/{\delta_1} + {\lambda_2^2}/{\delta_2}$, it holds that
	\begin{align}\label{E:proof_theo_main_100}
	& \delta_B \beta^2 \sum_{j=1}^p \sum_{k=1}^K w_k \Vert X^{(k)} (\hat{a}_j^{(k)} - a_{0j\hat{\pi}}^{(k)}) \Vert_2^2 / n_k \!+\! \delta_W \sum_{k=1}^K w_k \Vert \hat{\Omega}^{(k)} - \Omega_{0\hat{\pi}}^{(k)} \Vert_F^2 + \lambda^2 \delta_s \vert \hat{\G} \vert  \nonumber
	\\
	& \hspace{2cm} \leq \lambda^2 \vert \G_0^\un \vert + \frac{\lambda_2^2 (p + \vert \G_{0\hat{\pi}}^\un \vert)}{\delta_2} + \frac{\lambda_1^2 \vert \G_{0\hat{\pi}}^\un \vert}{\delta_1}.
	\end{align}
	Note that compared with Lemma~\ref{lem:bounds}, we have replaced $\sum_{j=1}^p \left(\frac{\hat{\omega}_j^{(k)} - {\omega_0}_j^{(k)}}{\hat{\omega}_j^{(k)}} \right)^2$ by $\Vert \hat{\Omega}^{(k)} - \Omega_{0\hat{\pi}}^{(k)} \Vert_F^2$. This follows from combining the facts that $\sum_{j=1}^p (\hat{\omega}_j^{(k)} - {\omega_0}_j^{(k)})^2$ is equal to $\Vert \hat{\Omega}^{(k)} - \Omega_{0\hat{\pi}}^{(k)} \Vert_F^2$ and that $\frac{1}{\hat{\omega}_j^{(k)}}$ is bigger than $1/\beta^2$ on the random event $\Ev_3$. 
	In addition, the constants $\beta, \sigma_0$, $\delta_1$ and $\delta_2$ in Lemma~\ref{lem:bounds} have been absorbed into the new constants $\delta_B$ and $\delta_W$. 
	We also replaced $\lambda^2 - \frac{\lambda_1^2}{\delta_1} - \frac{\lambda_2^2}{\delta_2}$ by $\lambda^2 \delta_s$ for some $0 < \delta_s < 1$. 
	
	By applying Lemma~\ref{lm:lambda03} [cf.~\eqref{eq:frobenius}] we may bound the first summand on the left-hand side of \eqref{E:proof_theo_main_100} by $\delta_B \sum_{k=1}^K w_k \Vert \hat{A}^{(k)} - A_{0\hat{\pi}}^{(k)} \Vert_F^2$. Furthermore, replacing $\lambda_1$ and $\lambda_2$ by some $\lambda_0$ that scales as $\lambda_0^2 \asymp \frac{\log p}{n} (p / \vert \G_0^\un\vert \vee 1)$, we obtain that
	\begin{align} \label{eq:bound}
	\delta_B \sum_{k=1}^K w_k \Vert \hat{A}^{(k)} - A_{0\hat{\pi}}^{(k)} \Vert_F^2 + & \delta_W \sum_{k=1}^K w_k \Vert \hat{\Omega}^{(k)} - \Omega_{0\hat{\pi}}^{(k)} \Vert_F^2 + \lambda^2 \delta_s \vert \hat{\G} \vert \nonumber \\
	& \leq \lambda^2 \vert \G_0^\un \vert + \lambda_0^2 \vert \G_{0\hat{\pi}}^\un \vert.
	\end{align}
	By applying~\eqref{eq:bound}, we have that for a constant $\eta_0$ small enough, \eqref{E:statement_theo_810} and~\eqref{E:statement_theo_820} in Theorem~\ref{thm:betamin} hold by choosing $\lambda$ such that $\lambda^2 \asymp \frac{\log p}{n} (p / \vert \G_0^\un\vert \vee 1)$ and $\lambda^2 \delta_s > \lambda_0^2 \delta_s'$. Moreover, from \eqref{E:statement_theo_810} we further infer that 
	\begin{align} \label{eq:bound2}
	\delta_B \sum_{k=1}^K w_k \Vert \hat{A}^{(k)} - A_{0\hat{\pi}}^{(k)}\Vert_F^2 + \lambda^2 \delta_s'' \vert \hat{\G} \vert \leq \lambda^2 \vert \G_0^\un \vert,
	\end{align}
	where the constant $\delta_s''$ is chosen such that $\lambda^2 \delta_s'' = \lambda^2 \delta_s - \lambda_0^2 \delta_s'$ in \eqref{E:statement_theo_810}.
	From \eqref{eq:bound2} it can thus be inferred that $\vert \hat{\G}\vert \leq \vert \G_0^\un \vert/\delta_s''$. Combining this with \eqref{E:statement_theo_820} and the fact that $\vert \G_{0\hat{\pi}}^\un \vert \geq c_g' / c_g \vert \G_0^\un \vert$, we recover the first part of \eqref{E:edges_main_theorem} in the statement of the theorem, i.e., $|\hat{\G}| \asymp |\G_{0\hat{\pi}}^\un|$. For the relation between $\vert \G_{0\hat{\pi}}^\un \vert$ and $\vert \G_0^\un \vert$, we use that $| \G_{0\hat{\pi}}^\un | \leq |\hat{\G}| / c_g \leq \vert \G_0^\un \vert/ (\delta_s'' \cdot c_g)$ and $\vert \G_{0\hat{\pi}}^\un \vert \geq c_g' / c_g \vert \G_0^\un \vert$.
	Finally, to recover \eqref{E:convergence_main_theorem} we combine \eqref{eq:bound} with \eqref{E:edges_main_theorem}, which concludes the proof. \hfill\qed
	
	\subsection{Proof of Theorem~\ref{thm:relax}} \label{sec:relaxproof}
	
	We first introduce a lemma that will be instrumental in proving Theorem~\ref{thm:relax} and that can be obtained directly from Lemmas 7.2 and 7.3 in \cite{vandegeer2013}.
	
	\begin{lemma}
		\cite[Lemmas 7.2 and 7.3]{vandegeer2013}
		\label{lm:oribetamin}
		Suppose for some $\delta_B, \delta_s, \lambda_0, \lambda > 0$ one has that $\delta_B \Vert \hat{A}^{(k)} - A_{0\hat{\pi}}^{(k)} \Vert_F^2 + \lambda^2 \delta_s \vert \hat{\G}^{(k)} \vert \leq \lambda^2 \vert \G_0^{(k)} \vert + \lambda_0^2 \vert \G_{0\hat{\pi}}^{(k)} \vert$. Let $\tilde{\lambda}^2 \delta_B \geq \lambda^2 + \lambda_0^2$ and assume that $\sum_{i,j} \mathbf{1} \left\{ \left| [A_{0\hat{\pi}}^{(k)}]_{i,j} \right| \geq \tilde{\lambda} / \eta_2 \right\} \geq (1 - \eta_1) \vert \G_{0\hat{\pi}}^{(k)}\vert$.
		Then
		\begin{align*}
		\delta_B \Vert \hat{A}^{(k)} - A_{0\hat{\pi}}^{(k)} \Vert_F^2 + \left( \lambda^2 \delta_s - \frac{\lambda_0^2}{1 - \eta_1 - \eta_2^2}\right) \vert \hat{\G}^{(k)} \vert \leq \lambda^2 \vert \G_0^{(k)} \vert
		\end{align*}
		and
		\begin{align*}
		\vert \hat{\G}^{(k)} \vert \geq (1 - \eta_1 - \eta_2^2) \vert \G_{0\hat{\pi}}^{(k)} \vert \geq (1 - \eta_1 - \eta_2^2) \vert \G_0^{(k)} \vert.
		\end{align*}
	\end{lemma}
	
	In order to show Theorem~\ref{thm:relax}, we begin the proof just like for Theorem~\ref{thm:l0} in Section~\ref{sec:proof_mainthm} until we get to expression \eqref{eq:bound}.
	It then follows from Condition~\ref{cd7}' and $|\hat{\G}| \geq \sum_{k=1}^K w_k|\hat{\G}^{(k)}|$ that there exists some $\lambda_0'^2 \asymp C_{\max} \frac{\log p}{n} (p / \vert \G_0^\un\vert \vee 1)$, where $C_{\max}$ is defined in Condition~\ref{cd8}, such that for any $\lambda >0$,
	\begin{align*}
	\begin{split}
	& \delta_B \sum_{k=1}^K w_k \Vert \hat{A}^{(k)} - A_{0\hat{\pi}}^{(k)} \Vert_F^2 + \delta_W \sum_{k=1}^K w_k \Vert \hat{\Omega}^{(k)} - \Omega_{0\hat{\pi}}^{(k)} \Vert_F^2 + \lambda^2 \delta_s \sum_{k=1}^K w_k \vert \hat{\G}^{(k)} \vert \\
	& \quad \leq \lambda^2 c_t(\pi_0) \sum_{k=1}^K w_k \vert \G_0^{(k)} \vert + \lambda_0'^2 \sum_{k=1}^K w_k \vert \G_{0\hat{\pi}}^{(k)}\vert.
	\end{split}
	\end{align*}
	Let $\lambda'^2 := \lambda^2 \cdot c_t(\pi_0)$ and $\delta_s' := \delta_s / c_t(\pi_0)$, it then follows that there must exist at least one $k$ such that for any $\lambda' >0$,
	\begin{align*} 
	\delta_B \Vert \hat{A}^{(k)} - A_{0\hat{\pi}}^{(k)} \Vert_F^2 + \delta_W \Vert \hat{\Omega}^{(k)} - \Omega_{0\hat{\pi}}^{(k)} \Vert_F^2 + \lambda'^2 \delta_s' \vert \hat{\G}^{(k)} \vert \leq \lambda'^2 \vert \G_0^{(k)} \vert + \lambda_0'^2 \vert \G_{0\hat{\pi}}^{(k)}\vert
	\end{align*}
	Since according to Condition~\ref{cd7}', $c_t(\pi)$ scales as a constant for permutations consistent with $\G_0^\un$,  we have that $\delta_s'$ is still a constant and $\lambda' \asymp \lambda$. In this case, it follows from Lemma~\ref{lm:oribetamin} and Condition~\ref{cd8}' that there exists some constant $0 < \delta_s < 1$ and $\delta_s' > 0$ such that by choosing $\lambda'$ such that $\lambda' \asymp C_{\max} \frac{\log p}{n} (p / \vert \G_0^\un\vert \vee 1)$ and $\lambda'^2 \delta_s > \lambda_0'^2 \delta_s'$, it holds that
	\begin{align*}
	\delta_B \Vert \hat{A}^{(k)} - A_{0\hat{\pi}}^{(k)}\Vert_F^2 + \left(
	\lambda'^2 \delta_s - \lambda_0'^2 \delta_s' \right) \vert \hat{\G}^{(k)} \vert \leq \lambda'^2 \vert \G_0^{(k)} \vert.
	\end{align*}
	It also follows from Lemma~\ref{lm:oribetamin} that $\vert \hat{\G}^{(k)} \vert \geq c_g \vert \G_{0\hat{\pi}}^{(k)} \vert \geq c_g \vert \G_0^{(k)} \vert$ for some positive constant $c_g$.
	Mimicking the arguments employed in the proof of Theorem~\ref{thm:l0} from \eqref{eq:bound2} until the end of the proof, one can show that expressions \eqref{E:convergence_main_theorem_relaxed} and \eqref{E:edges_main_theorem_relaxed} in the statement of Theorem~\ref{thm:relax} hold true, which completes the proof. \hfill\qed

\subsection{Proof of Corollary~\ref{cor:intl0}} \label{sec:corproof}

The following lemma is instrumental in proving the corollary.

\begin{lemma}\label{lem:cor1} \cite[Lemma 4]{hauser2015jointly} Given fixed $\G$, the maximum likelihood estimator in~\eqref{obj:jointl0_intervened} can be written as
	\begin{align*}
	p + \sum_{j=1}^p \Bigg( \underset{a \in \R^{\vert{\pa}_j(\G) \vert}}{\min} \frac{n_{-j}}{n} \log \Bigg( \sum_{k : j \not\in I_k} \frac{n_k}{ n_{-j}}\Vert \hat{X}_j^{(k)} - \hat{X}^{(k)}_{{\pa}_{j}(\G)} \cdot a \Vert_2^2 / n_k \Bigg) \\
	+ \sum_{k : j \in I_k} w_k \log \left( \Vert \hat{X}_j^{(k)} \Vert_2^2 / n_k \right)\Bigg)
	\end{align*}
	where $n_{-j}$ is the total number of samples where node $j$ is not intervened on, i.e., $n_{-j} = \sum_{k : j \not\in I_k} n_k$.
\end{lemma}

Recall that in the interventional setting $\G_0^\un$ is given by the true graph $\G_0$ of the non-intervened model, and that the $K$ models $(A_0^{(k)}, \Omega_0^{(k)})$ to be inferred correspond to the interventional models $(A_0^{I_k}, \Omega_0^{I_k})$. Denoting by $(\hat{\pi}, \hat{A}, \hat{\Omega})$ the (non-intervened) global optimum of \eqref{obj:jointl0_intervened}, let $\hat{\omega}_j^{(k)}$ denote the empirical variance of the random variable $X_j^{(k)} - X^{(k)} \hat{a}_j$ if $j \in I_k$ and the empirical variance of $X_j^{(k)}$ otherwise. It follows from Lemma~\ref{lem:cor1} that the global optimum satisfies
		\begin{align*}
		&p + \sum_{j=1}^p \left( \frac{n_{-j}}{n} \log \left( \sum_{k : j \not\in I_k} \frac{n_k}{ n_{-j}} \hat{\omega}_j^{(k)} \right) + \sum_{k : j \in I_k} w_k \log \hat{\omega}_j^{(k)} \right) + \lambda^2 \vert \hat{\G} \vert \\
		&\hspace{2cm}\leq \sum_{j=1}^p \sum_{k=1}^K w_k \frac{\Vert \hat{\epsilon}_{j\hat{\pi}}^{(k)} \Vert_2^2 / n_k}{\omega_{0j\hat{\pi}}^{(k)}} + \sum_{j=1}^p \sum_{k=1}^K w_k \log \omega_{0j\hat{\pi}}^{(k)} + \lambda^2 \vert \G_0^\un \vert.
		\end{align*}
		
Then, applying the inequality $\log (\sum_{k=1}^K w_k a_k) \geq \sum_{k=1}^K w_k  \log a_k$ for any choices of $a_1, \ldots, a_K > 0$ and $w_1, \dots,  w_K > 0$ with $\sum_{k=1}^K w_k = 1$, we obtain
		\begin{align*}
		p + \sum_{j=1}^p \sum_{k=1}^K w_k \log \hat{\omega}_j^{(k)} + \lambda^2 \vert \hat{\G} \vert
		\leq & \sum_{j=1}^p \sum_{k=1}^K w_k \frac{\Vert \hat{\epsilon}_{j\hat{\pi}}^{(k)} \Vert_2^2 / n_k}{\omega_{0j\hat{\pi}}^{(k)}} \\
		& + \sum_{j=1}^p \sum_{k=1}^K w_k \log \omega_{0j\hat{\pi}}^{(k)} + \lambda^2 \vert \G_0^\un \vert.
		\end{align*}
Hence Corollary~\ref{cor:intl0} directly follows from the proof of Theorem~\ref{thm:l0}. \hfill\qed

\end{appendix}

\bibliographystyle{plain}

\bibliography{main}

\end{document}